\documentclass[reqno]{amsart}

\usepackage{graphicx} 
\usepackage{hyperref}

\usepackage[mathscr]{euscript}
\usepackage{amsmath}
\usepackage{amssymb}
\usepackage{pstricks}
\usepackage{enumerate}

\usepackage{graphicx}
\usepackage{verbatim}
\usepackage{bm}

\usepackage{slashed}
\usepackage{mathrsfs}

\usepackage{bbm}
\usepackage{bm}

\newtheorem{MainThm}{Theorem}

\theoremstyle{plain}

\newtheorem{thm}[equation]{Theorem}
\newtheorem{cor}[equation]{Corollary}
\newtheorem{lem}[equation]{Lemma}
\newtheorem{prop}[equation]{Proposition}

\theoremstyle{definition}
\newtheorem{definition}[equation]{Definition}

\newtheorem{definition-proposition}[equation]{Definition/Proposition}

\newtheorem{defn}[equation]{Definition}

\theoremstyle{remark}
\newtheorem{remark}[equation]{Remark}

\newtheorem{rem}[equation]{Remark}
\newtheorem{rems}[equation]{Remarks}

\numberwithin{equation}{subsection}

\newcommand\mnote[1]{\marginpar{\tiny #1}}

\newcommand{\bC}{\mathbb{C}}

\newcommand{\bKK}{\mathbf{KK}}

\newcommand{\bN}{\mathbb{N}}

\newcommand{\bQ}{\mathbb{Q}}
\newcommand{\bR}{\mathbb{R}}
\newcommand{\bS}{\mathbf{S}}

\newcommand{\bZ}{\mathbb{Z}}

\newcommand{\gA}{\bold{A}}
\newcommand{\gB}{\bold{B}}
\newcommand{\gC}{\bold{C}}
\newcommand{\gD}{\bold{D}}

\newcommand{\gR}{\bold{R}}

\newcommand{\gb}{\mathbf{b}}
\newcommand{\gx}{\mathbf{x}}
\newcommand{\gy}{\mathbf{y}}
\newcommand{\ga}{\mathbf{a}}

\newcommand{\gGamma}{\mathbf{\Gamma}}

\newcommand{\gu}{\bold{u}}
\newcommand{\gs}{\bold{s}}

\newcommand{\thom}{\tau}

\newcommand{\cI}{\mathcal{I}}
\newcommand{\cJ}{\mathcal{J}}

\newcommand{\cL}{\mathcal{L}}

\newcommand{\cR}{\mathcal{R}}

\newcommand{\gt}{\mathbf{t}}

\newcommand{\an}{\mathrm{an}}

\newcommand{\bott}{\beta}

\newcommand{\Cl}{\mathbf{Cl}}

\newcommand{\cstar}{\mathrm{C}^{\ast}}
\newcommand{\cstarred}{\mathrm{C}^{\ast}_{\mathrm{r}}}
\newcommand{\cstarmax}{\mathrm{C}^{\ast}_{\mathrm{m}}}

\newcommand{\Dir}{\slashed{\mathfrak{D}}}

\newcommand{\eps}{\epsilon}

\newcommand{\GL}{\mathrm{GL}}

\newcommand{\komprod}{\#}
\newcommand{\exprod}{\boxtimes}
\newcommand{\vol}{\mathrm{vol}}

\newcommand{\Free}{F}

\newcommand{\SB}{\mathrm{SB}}

\newcommand{\ind}{\operatorname{ind}}
\newcommand{\id}{\operatorname{id}}

\newcommand{\inddiff}{\operatorname{inddiff}} 

\newcommand{\Sub}{\mathrm{Sub}}
\newcommand{\Epi}{\mathrm{Epi}}
\newcommand{\Fr}{\mathrm{Fr}}

\newcommand{\BC}{\mu}
\newcommand{\Nov}{\nu}

\newcommand{\KO}{\mathbf{KO}}
\newcommand{\ko}{\mathbf{ko}}

\newcommand{\loopinf}[1]{\Omega^{\infty #1}}

\newcommand{\MT}{\mathbf{MT}}

\newcommand{\MSpin}{\mathbf{MSpin}}

\newcommand{\norm}[1]{\| #1 \|}
\newcommand{\normalize}[1]{\frac{#1}{(1+{#1}^{2})^{1/2}}}

\newcommand{\per}{\mathrm{per}}

\newcommand{\Riem}{\cR}
\newcommand{\res}{\mathrm{res}}

\newcommand{\Spin}{\mathrm{Spin}}

\newcommand{\spinor}{\slashed{\mathfrak{S}}}
\newcommand{\scal}{\mathrm{scal}}

\newcommand{\Th}{\mathbf{Th}}
\newcommand{\tor}{\mathrm{tor}}

\newcommand\lra{\longrightarrow}

\newcommand\Diff{\mathrm{Diff}}

\newcommand\hocolim{\operatorname*{hocolim}}
\newcommand\colim{\operatorname*{colim}}

\newcommand{\round}{\circ}

\newcommand\hcoker{/\!\!/}

\newcommand{\inter}[1]{\mathrm{int}(#1)}

\newcommand{\hAut}{\mathrm{hAut}}

\newcommand{\CircNum}[1]{\ooalign{\hfil\raise .00ex\hbox{\scriptsize #1}\hfil\crcr\mathhexbox20D}}

\newcommand{\hq}{/\!\!/}

\newcommand{\SE}{\mathrm{SE}}

\newcommand{\kothom}[1]{\lambda_{#1}}

\newcommand{\scpr}[1]{\langle #1 \rangle}

\renewcommand{\KO}{\mathsf{KO}}
\renewcommand{\ko}{\mathsf{ko}}

\renewcommand{\MT}[1]{\mathsf{MT #1}}
\renewcommand{\MSpin}{\mathsf{MSpin}}

\renewcommand{\bKK}{\mathsf{KK}}
\renewcommand{\Th}{\mathsf{Th}}

\renewcommand{\cstar}{\mathbf{C}^{\ast}}
\renewcommand{\cstarred}{\mathbf{C}^{\ast}_{\mathrm{r}}}
\renewcommand{\cstarmax}{\mathbf{C}^{\ast}_{\mathrm{m}}}

\usepackage[all]{xy}
\SelectTips{cm}{10}
\CompileMatrices

\title[Infinite loop spaces and positive scalar curvature. II.]{Infinite loop spaces\\and positive scalar curvature\\in the presence of a fundamental group}

\author{Johannes Ebert}
\email{johannes.ebert@uni-muenster.de}
\address{
Mathematisches Institut der Westf{\"a}lischen Wilhelms-Universit{\"a}t M{\"u}nster\\
Einsteinstr. 62\\
DE-48149 M{\"u}nster\\
Germany
}

\author{Oscar Randal-Williams}
\email{o.randal-williams@dpmms.cam.ac.uk}
\address{
Centre for Mathematical Sciences\\
Wilberforce Road\\
Cambridge CB3 0WB\\
UK}

\date{\today}

\begin{document}

\begin{abstract}
This is a continuation of our previous work with Botvinnik on the nontriviality of the secondary index invariant on spaces of metrics of positive scalar curvature, in which we take the fundamental group of the manifolds into account. We show that the secondary index invariant associated to the vanishing of the Rosenberg index can be highly nontrivial, for positive scalar curvature Spin manifolds with torsionfree fundamental groups which satisfy the Baum--Connes conjecture. For example, we produce a compact Spin 6-manifold such that its space of positive scalar curvature metrics has each rational homotopy group infinite dimensional.

At a more technical level, we introduce the notion of ``stable metrics'' and prove a basic existence theorem for them, which generalises the Gromov--Lawson surgery technique, and we also give a method for rounding corners of manifold with positive scalar curvature metrics.
\end{abstract}

\maketitle

\tableofcontents

\section{Introduction and statement of results}

It is a simple observation that if $M$ is a closed manifold, then the product $M \times S^2$ admits a metric of positive scalar curvature, so the fundamental group by itself not an obstruction against the existence of a metric of positive scalar curvature (at least in the high-dimensional regime). Nevertheless, the fundamental group has played an important role in the theory of positive scalar curvature since the work of Gromov and Lawson \cite{GromLawpi1, GL2}, the point being that there are obstructions to the existence of psc metrics on $M$ which lie in abelian groups depending on $\pi_1 (M)$. The most important is the \emph{Rosenberg index} \cite{RosNovI} which is defined for a closed $d$-dimensional spin manifold $M$. It is defined in terms of the real $K$-theory of the reduced group $\mathrm{C}^*$-algebra $\cstarred(\pi_1(M))$, as an element $\alpha_{\mathrm{r}} (M) \in KO_d (\cstarred(\pi_1(M)))$. 

If $M$ has a psc metric then the Rosenberg index vanishes, but in an appropriate sense the positivity of the scalar curvature also gives a \emph{reason} for it to vanish: therefore, following the recipe described in \cite[\S 3.3.2]{BERW}, if $\Riem^+ (M)$ denotes the space of psc metrics on $M$ then there is an associated secondary index invariant
\[
\inddiff^{\pi_1 (M)} : \Riem^+ (M) \times \Riem^+ (M) \lra  \loopinf{+d+1} \KO (\cstarred(\pi_1(M)))
\]
to an appropriate infinite loop space of the real $K$-theory spectrum of $\cstarred(\pi_1 (M))$. More generally, for a spin manifold $W$ with collared boundary, one may fix a metric $g \in \cR^+(\partial W)$ and consider the space $\Riem^+ (W)_g$ of psc metrics on $W$ which are of the form $g+dt^2$ on the collar. Then, given a discrete group $G$ and a map $f:W \to BG$, there is a map 
\[
\inddiff^{G} : \Riem^+ (W)_g \times \Riem^+ (W)_g \lra \loopinf{+d+1} \KO (\cstarred(G))
\]
(which depends on the reference map $f$). Usually, we fix a basepoint $h_0\in \Riem^+ (W)_g$ and consider the map $\inddiff^G_{h_0}(\_) = \inddiff^G (h_0,\_)$. On homotopy groups, $\inddiff^G_{h_0}$ induces a map
\[
 \inddiff^G_{h_0}: \pi_k (\Riem^+ (W),h_0) \lra \pi_k (\loopinf{+d+1} \KO (\cstarred (G)) = KO_{d+k+1} (\cstarred(G))
\]
to the $KO$-theory groups of $\cstarred (G)$.

In previous work with Botvinnik \cite{BERW} we studied the nontriviality of this map when $G=1$, in which case the target is just the real topological $K$-theory of a point. Our goal in this paper is to extend this to more general groups $G$, when the target can be much richer.

To state our results, we must introduce some notation from stable homotopy theory. The Madsen--Tillmann--Weiss spectrum associated with the map $B \Spin (d)\to BO(d)$ is denoted $\MT{Spin(d)}$. The Atiyah--Bott--Shapiro orientation of spin vector bundles yields a spectrum map 
$\kothom{-d}: \MT{Spin(d)} \to \Sigma^{-d} \KO$. For a discrete group $G$, there is the Novikov assembly map $\Nov: \KO\wedge BG_+ \to \KO (\cstarred(G))$ which we will recall in Definition \ref{defn:novikov-assembly} below.

\begin{MainThm}\label{main-theorem-mapintoRiem}
Let $W^{2n}$ be a compact connected spin manifold with boundary $\partial W$, let $G$ be a discrete group and let $f: W \to BG$ be a map. Furthermore, let $g \in \Riem^+ (\partial W)$ and $h_0 \in \Riem^+(W)_g$. Assume that
\begin{enumerate}[(i)]
\item $n \geq 3$, and
\item the homomorphism $\pi_1(f):\pi_1 (W) \to \pi_1 (BG)=G$ is split surjective.
\end{enumerate}
Then there exists a map $\Psi: \loopinf{+1} ( \MT{Spin(2n)} \wedge BG_+) \to \Riem^+ (W)_g$ such that the following diagram commutes up to homotopy:
\[
\xymatrix{
\loopinf{+1} (\MT{Spin (2n)} \wedge BG_+ )\ar[rr]^-{\Psi} \ar[d]_{\loopinf{+1} \kothom{-2n} \wedge \id} && \Riem^+ (W)_g \ar[d]^{\inddiff^G_{h_0}}\\
\loopinf{+2n+1} (\KO \wedge BG_+) \ar[rr]^-{\loopinf{+2n+1} \Nov} & & \loopinf{+2n+1} \KO (\cstarred (G)).
}
\]
\end{MainThm}

\begin{rems}\label{rem:ThmA}\mbox{}
\begin{enumerate}[(i)]
\item The case $G=1$ of Theorem \ref{main-theorem-mapintoRiem} is precisely our previous result with Botvinnik \cite[Theorem B]{BERW}. 
\item In \cite{BERW}, we also proved a similar result for odd-dimensional manifolds, derived from \cite[Theorem B]{BERW}. The analogous result for nontrivial $G$ can be derived from Theorem \ref{main-theorem-mapintoRiem}, and this will be done in the forthcoming M\"unster PhD thesis of Lukas Buggisch. See Remark \ref{rem:spectralflow} for more details. 
\item There is also a maximal group $\mathrm{C}^*$-algebra $\cstarmax(G)$, with a homomorphism $\omega:\cstarmax(G) \to \cstarred(G)$. The secondary Rosenberg index can be made to have target $\loopinf{+2n+1} \KO (\cstarmax (G))$, and the Novikov assembly map naturally factors through a maximal version $\Nov_\mathrm{m}: \KO\wedge BG_+ \to \KO (\cstarmax(G))$. Theorem \ref{main-theorem-mapintoRiem} in fact holds when the bottom right-hand corner is replaced with $\loopinf{+2n+1} \KO (\cstarmax (G))$. Although this is a stronger theorem, we prefer to state it in terms of the reduced group $\mathrm{C}^*$-algebra as it is this to which the Baum--Connes conjecture applies, as we will now discuss.
\end{enumerate} 
\end{rems}

\subsection{Applications}

To draw computational consequences out of Theorem \ref{main-theorem-mapintoRiem} as in \cite[\S 5]{BERW}, one needs another assumption, namely that $G$ is torsionfree and satisfies the Baum--Connes conjecture, i.e. that the Baum--Connes assembly map
$$\BC : KO^G_*(\underline{E}G) \lra KO_*(\cstarred(G))$$
is an isomorphism. For torsionfree $G$, we may identify $KO^G_*(\underline{E}G)$ with $KO_*(BG)$ and the Baum--Connes assembly map with the map induced by the Novikov assembly map $\nu$ on homotopy groups. The Baum--Connes conjecture predicts that $\BC$ is an isomorphism, so for torsionfree $G$ it predicts that the Novikov assembly map is a weak homotopy equivalence. The Baum--Connes conjecture has been proven for vast classes of groups; see \cite[\S 2.6]{LR} for a slightly outdated survey.

\subsubsection{Bott-stabilised integral surjectivity}

Let $B^8$ be a Bott manifold, that is a Spin 8-manifold having $\widehat{\mathscr{A}}(B) = \beta \in KO_8(*)$. By the work of Joyce \cite[\S 6]{Joyce}, we may choose such a $B$ having a metric $g_B$ of holonomy $\mathrm{Spin}(7)$, which is then Ricci-flat and hence scalar-flat. For a closed manifold $M$ there are then induced maps
$$\cR^+(M) \overset{\_ \times (B, g_B)}\lra \cR^+(M \times B) \overset{\_ \times (B, g_B)}\lra \cR^+(M \times B \times B) \lra \cdots$$
and we write $\cR^+(M)[B^{-1}]$ for the homotopy colimit. As $\widehat{\mathscr{A}}(B)$ acts invertibly on $\KO(C^*(G))$, if $M$ is Spin and $f : M \to BG$ is a reference map then there is an extension
\begin{equation}\label{eq:StableRosenbergIndex}
\inddiff^G_{h_0}[B^{-1}] : \cR^+(M)[B^{-1}] \lra \Omega^{\infty+d+1} \KO(\cstarred(G))
\end{equation}
of the secondary Rosenberg index map for $M$.

\begin{MainThm}\label{thm:StableSurj}
If $G$ is a torsionfree group satisfying the Baum--Connes conjecture, $f_* : \pi_1(M) \to G$ is split surjective, and $d=\mathrm{dim}(M)$ is even\footnote{This hypothesis may be dropped once the work of Remark \ref{rem:ThmA} (ii) is complete.}, then the map \eqref{eq:StableRosenbergIndex} is surjective on all homotopy groups.
\end{MainThm}

\subsubsection{Rational surjectivity}

The spectrum map $\kothom{-2n}$ factors through the desuspension of the connective $KO$-theory spectrum as 
$$\kothom{-2n}: \MT{Spin (2n)} \overset{\kothom{-2n}'}\lra \Sigma^{-2n} \ko \overset{\Sigma^{-2n}\per}\lra \Sigma^{-2n} \KO,$$
and hence the spectrum map $\Nov \circ (\kothom{-2n} \wedge \id)$ may be written as the composition
\begin{align*}
\MT{Spin (2n)} \wedge BG_+ \stackrel{\kothom{-2n}'\wedge \id}{\lra} \Sigma^{-2n} \ko \wedge BG_+ &\stackrel{\Sigma^{-2n}\per\wedge \id}{\lra} \Sigma^{-2n} \KO \wedge BG_+\\
& \stackrel{\Sigma^{-2n}\Nov}{\lra} \Sigma^{-2n} \KO(\cstarred(G)).  
\end{align*}
It follows that
$$\mathrm{Im}(\pi_r (\inddiff^G_{h_0})) \supset \mathrm{Im}(\pi_{r+1}(\Sigma^{-2n}\Nov \circ \Sigma^{-2n}(\per \wedge \id) \circ \kothom{-2n}'\wedge \id)).$$

By a standard characteristic class computation (see e.g. \cite[Theorem 5.2.1]{BERW}), the map $\pi_{r+1}(\kothom{-2n}')\otimes \bQ$ is surjective, and by the (collapsing) Atiyah--Hirzebruch spectral sequence it follows that $\pi_{r+1} (\kothom{-2n}' \wedge \id)\otimes \bQ$ is too. If $G$ is torsionfree and satisfies the Baum--Connes conjecture then as we explained above the map $\Nov$ is a weak equivalence. The map $\per \wedge \id:\ko \wedge BG_+ \to \KO \wedge BG_+$ is not rationally surjective, but if the rational homological dimension of the group $G$ is finite, say equal to $ q$, then 
$$(\per \wedge \id)_* : \pi_s (\ko\wedge BG_+ )\otimes \bQ \lra \pi_s (\KO\wedge BG_+ )\otimes \bQ$$
is surjective for $s \geq q$. From these facts, we derive the following.

\begin{MainThm}\label{main-theorem-rational-homotopy}
Let $W$ be a spin manifold of dimension $2n \geq 6$, possibly with boundary. Let $G$ be a group and let $f: W \to BG$ be a map such that $\pi_1 (f): \pi_1(W) \to G$ is split surjective. Assume that $h_0 \in \Riem^+ (W)_g \neq \emptyset$, and that
\begin{enumerate}[(i)]
 \item $G$ satisfies the (rational) Baum--Connes conjecture,
 \item $G$ is torsionfree and has finite rational homological dimension $q$, and
 \item $r \geq q-2n-1$.
\end{enumerate}
Then the image of
\[
(\inddiff^G_{h_0})_*: \pi_r (\Riem^+ (W)_g) \lra KO_{2n+1+r} (C^*(G)) \otimes \bQ
\]
generates the target as a $\bQ$-vector space.
\end{MainThm}

\begin{remark}
Some rational consequences can also be obtained without the torsionfree hypothesis. When $G$ is a finite group the map
$$KO_*(BG) \otimes \bQ = KO^G_*(EG)\otimes \bQ \lra KO^G_*(\underline{E}G)\otimes \bQ$$
is split injective (as $KO_*(BG)\otimes \bQ = \bQ$, so a splitting may be given by the augmentation of the real representation ring of $G$), and a spectral sequence argument in equivariant $KO$-theory show that the analogous map is split injective for any group $G$, see \cite[Corollary A.3]{Matthey}. Thus if $G$ satisfies the (rational injectivity part of the) Baum--Connes conjecture then the Novikov assembly map is also rationally injective. Thus, in the situation of Theorem \ref{main-theorem-rational-homotopy} but omitting the torsionfree hypothesis, for $r \geq 2$ one finds that $\pi_r (\Riem^+ (W)_g) \otimes \bQ$ is at least as large as $KO_{2n+1+r}(BG) \otimes \bQ$ (we leave formulating the awkward conclusion of this argument for $r < 2$ to the reader).

On the other hand, consider the homotopy fibre $\mathsf{F}$ of the Novikov assembly map. On homotopy groups, the fibre sequence $\mathsf{F} \to \KO \wedge BG_+ \stackrel{\Nov}{\to} \KO(\cstarred(G))$ induces a long exact sequence, which may be identified with the analytical surgery sequence of Higson and Roe \cite{HigsonRoeEta}
\[
\cdots \lra  KO_d (BG) \overset{\nu_*}\lra KO_d (\cstarred (G)) \stackrel{\partial}{\lra} \mathrm{S}_{d-1}(G) \lra \cdots.
\]
Recently, B{\'a}rcenas and Zeidler \cite[Corollary 1.5]{BarcenasZeidler} have shown that the image of 
\[
 \pi_0 (\Riem^+ (M^d)) \stackrel{\inddiff^{G}_{h_0}}{\lra} KO_{d+1}(\cstarred(G)) \stackrel{\partial}{\lra} \mathrm{S}_d (G) \lra \mathrm{S}_d(G) \otimes \bQ 
\]
generates the target group, if $G$ satisfies the Baum--Connes conjecture and has rational homological dimension at most $2$ (e.g. if $G$ is finite). This supersedes previously known results by Botvinnik--Gilkey \cite{BotGil}, Piazza--Schick \cite{PS} and Weinberger--Yu \cite{WY} (which were all for odd $d$). By virtue of our construction, we can only construct classes in the homotopy of $\Riem^+ (W)$ which become trivial in $\mathrm{S}_d (G)$. 
\end{remark}

\subsubsection{A large example}

Let $\Free_r$ denote the free group of rank $r$. The \emph{Stallings--Bieri} group $\SB_n$ is defined as the kernel of the homomorphism
$$\phi : (\Free_2)^n \lra \bZ$$
sending each generator to 1. This group is of type $(F_{n-1})$ (i.e., there exists a $K(\SB_n,1)$-complex with finite $(n-1)$-skeleton), but has $H_{n}(\SB_n;\bQ)$ countably infinite dimensional (see \cite[p.\ 37]{Bieri}). Furthermore, it has homological dimension $n$, being a subgroup of the $n$-dimensional group $(\Free_2)^n$. Thus the group
$$G = \SB_3 *  \SB_4 * \SB_5 * \SB_6$$
is torsionfree, has homological dimension $6$ and has type $(F_2)$ (so is finitely presented). The class of groups satisfying the Baum--Connes conjecture with coefficients is closed under (amalgamated) free products, finite direct products, passing to subgroups, and HNN extensions \cite[Theorem 5.2]{LR}, so the Baum--Connes conjecture holds for $G$. Therefore by the Atiyah--Hirzebruch spectral sequence we have that
$$KO_{7+r} (\cstarred(G)) \otimes \bQ \cong KO_{7+r}(BG) \otimes \bQ \cong \bigoplus_{i \geq 0} H_{7+r-4i}(BG;\bQ)$$
is a countably infinite dimensional vector space for each $r \geq 0$. 

Take a finite $2$-skeleton of $BG$, embed it into $\bR^5$, take a regular neighborhood $N$ and consider $M=\partial N$. This is a $4$-dimensional spin manifold with $\pi_1 (M)=G$. The 6-manifold $W=M \times S^2$ has a psc metric. 
This puts us in a position to apply Theorem \ref{main-theorem-rational-homotopy}, showing that $\pi_r (\Riem^+ (W))\otimes \bQ$ is countably infinite dimensional for $r \geq 2$, the same for $r=1$ after abelianising, and that $\pi_0 (\Riem^+ (W))$ maps onto an infinite rank abelian group. As a separable Fr{\'e}chet manifold (such as $\Riem^+(W)$) has the homotopy type of a countable CW complex \cite[Corollary 4]{Henderson} and hence countable homotopy groups, this is as large as homotopy groups of a space of psc metrics on a compact manifold can possibly be.

\subsection{Stable metrics: extension of the Gromov--Lawson surgery method}

The main technical tool we shall develop, which was not available to us in \cite{BERW}, is an extension of the Gromov--Lawson surgery method \cite{GL} and in particular of Chernysh's cobordism invariance theorem \cite{Chernysh} to the case of manifolds with boundaries. Let us first recall some notation and describe this result, which may be of interest in its own right.

If $W: M_0 \leadsto M_1$ is a cobordism with collared boundaries and $g_i \in \Riem^+ (M_i)$, $i=0,1$, we let $\Riem^+ (W)_{g_0,g_1}$ be the space of all metrics of positive scalar curvature on $W$ which are equal to $g_i + dt^2$ near $M_i$, with respect to the given collar. 
For any $h \in \Riem^+ (W)_{g_0,g_1}$, there are composition maps
\[
\mu (h, \_): \Riem^+ (V)_{g_1,g_2} \lra \Riem^+ (W \cup V)_{g_0,g_2}
\]
and 
\[
\mu ( \_,h): \Riem^+ (V')_{g_{-1},g_0} \lra \Riem^+ (V' \cup W)_{g_{-1},g_1}
\]
defined for all cobordisms $V: M_1 \leadsto M_2$, $V': M_{-1} \leadsto M_0$ and boundary conditions $g_{-1}\in \Riem^+ (M_{-1})$ and $g_2 \in \Riem^+ (M_2)$.

\begin{defn}\label{defn:right-stable}
Let $W: M_0 \leadsto M_1$ be a cobordism and let $h \in \Riem^+ (W)_{g_0,g_1}$. Then $h$ is called \emph{left-stable} if the map $\mu (\_, h): \Riem^+ (V)_{g_{-1},g_0} \to \Riem^+ (V \cup W)_{g_{-1},g_1}$ is a weak equivalence for all cobordisms $V:M_{-1}\leadsto M_0$ and all boundary conditions $g_{-1}$. Dually, $h$ is \emph{right-stable} if the map $\mu (h,\_): \Riem^+ (V)_{g_1,g_2} \to \Riem^+ (W \cup V)_{g_0,g_2}$ is a weak equivalence for all cobordisms $V: M_1 \leadsto M_2$ and all boundary conditions $g_2$. Finally, $h$ is \emph{stable} if it is both left- and right-stable.
\end{defn}

We prove two results about the existence of such metrics. The first shows that right-stable metrics exist in abundance.

\begin{MainThm}\label{thm:existence-stable-metrics1}
Let $d \geq 6$ and let $W: M_0 \leadsto M_1$ be a $d$-dimensional cobordism. Assume that the inclusion map $M_1 \to W$ is $2$-connected. Then for each $g_0 \in \Riem^+ (M_0)$, there is a $g_1 \in \Riem^+ (M_1)$ and a right-stable $h \in \Riem^+ (W)_{g_0,g_1}$. 
\end{MainThm}

Note that the metric $g_1 \in \Riem^+ (M_1)$ is part of the statement of the theorem, not of the hypothesis. In other words, the theorem does \emph{not} say that given $g_0, g_1$ such that $\Riem^+ (W)_{g_0 , g_1} \neq \emptyset$, then there is a metric $h \in \Riem^+ (W)_{g_0,g_1}$ which is right stable (let alone that \emph{any} such $h$ has these properties). The theorem applies to the case $M_0 = \emptyset$ (because the unique Riemannian metric on $\emptyset$ has positive scalar curvature). The main result of \cite{Chernysh} may be viewed as the special case $W= N \times D^k: \emptyset \leadsto N \times S^{k-1}$ of Theorem \ref{thm:existence-stable-metrics1} (for $k \geq 3$), but it took us some time to arrive at the formulation of Theorem \ref{thm:existence-stable-metrics1} as the correct generalisation of this result.

The second result shows that right-stable metrics are often automatically also left-stable. 

\begin{MainThm}\label{thm:right-and-left-stable}
Let $d \geq 6$ and let $W: M_0 \leadsto M_1$ be a $d$-dimensional cobordism. Assume that \emph{both} inclusion maps are $2$-connected. Then every right-stable $h \in \Riem^+ (W)_{g_0,g_1}$ is also left-stable.
\end{MainThm}

This theorem has no analogue in \cite{Chernysh}, as it does not hold in the situation considered there: a metric $h \in \Riem^+(W)_g$ on the cobordism $W= N \times D^k: \emptyset \leadsto N \times S^{k-1}$ being left-stable would mean that
$$\mu(\_, h) : \{*\} = \Riem^+ (\emptyset) \lra \Riem^+ (W)_{g}$$
is a weak equivalence, which is not generally true.
For later use we record the combination of Theorems \ref{thm:existence-stable-metrics1} and \ref{thm:right-and-left-stable} as follows.

\begin{cor}\label{cor:right-and-left-stable}
Let $d \geq 6$ and let $W: M_0 \leadsto M_1$ be a $d$-dimensional cobordism such that both inclusions $M_i \to W$ are $2$-connected. Then, given $g_0 \in \Riem^+ (M_0)$, there is a $g_1 \in \Riem^+ (M_1)$ and a stable $h\in \Riem^+ (W)_{g_0,g_1}$.\qed
\end{cor}

\begin{rem}
A related result has been proved by Walsh \cite[Theorem A]{Walsh4}, which in the language used here says that if $W$ is as in Corollary \ref{cor:right-and-left-stable} then for each $g_0 \in \Riem^+ (M_0)$, there is a $g_1 \in \Riem^+ (M_1)$ and a left-stable $h\in \Riem^+ (W)_{g_0,g_1}$. 
\end{rem}

To prove Theorems \ref{thm:existence-stable-metrics1} and \ref{thm:right-and-left-stable}, we introduce several technical tools, such as spaces of psc metrics on {manifolds with corners} and a procedure for rounding corners of manifolds with psc metrics. This is done in the fairly elementary Section \ref{sec:psc-spaces-on-mfds-w-corners}. The proof of Theorems \ref{thm:existence-stable-metrics1} and \ref{thm:right-and-left-stable} is given in Section \ref{sec:StableMetrics}.

\subsection{Outline of the proof of Theorem \ref{main-theorem-mapintoRiem}}

We follow the same general strategy as the argument for \cite[Theorem B]{BERW}. The proof is given in Sections \ref{sec:proof-constructivepart} and \ref{sec:indextheory}. In those sections, we shall assume familiarity with \cite[\S 2,3,4]{BERW} and focus on those parts of the argument which exhibit essential differences. The construction in \cite{BERW} rested on three pillars: Index theory, the Gromov--Lawson surgery method, and the high-dimensional Madsen--Weiss theorem of Galatius and the second named author.

Theorem \ref{main-theorem-mapintoRiem} has two distinct parts: one concerns the construction of $\Psi$, and the other the commutativity of the diagram. We have decided to separate these parts more cleanly than in \cite{BERW}. For the first part, the spin hypothesis does not play a role. Let $W^{2n}$ be a manifold and consider the Gauss map $\tau:W \to BO(2n)$. Let $W \stackrel{\ell}{\to} X \stackrel{\theta}{\to} BO(2n)$ be the second stage of the Moore--Postnikov tower of $\tau$. Recall that $\theta$ is a $2$-coconnected fibration and $\ell$ is a $2$-connected map. If $W$ is a spin manifold, then $X \simeq B\Spin (2n) \times B\pi_1 (M)$ (a complete classification of the second Moore--Postnikov stages of classifying maps of vector bundles can be found in \cite{Stolz2}). To the fibration $\theta$, we have the associated Madsen--Tillmann--Weiss spectrum $\MT{\theta}$ and, if $\partial W \to W$ is $2$-connected, there is a parameterised Pontrjagin--Thom map
\[
\alpha_W : B \Diff_\partial (W) \lra \loopinf{}_0 \MT{\theta}.
\]
The construction of the map $\Psi$ is a consequence of the following general theorem.

\begin{MainThm}\label{thm:factorization}
Let $W^{2n}$, $n \geq 3$, be a connected compact manifold viewed as a cobordism $W: \emptyset \leadsto \partial W$. Assume that $(W, \partial W)$ is $(n-1)$-connected. Let $g \in \Riem^+ (\partial W)$ be such that there exists a right stable $h \in \Riem^+ (W)_g$. Then there exists a fibration $p:T^+_{\infty}\to \loopinf{}_0 \MT{\theta}$ and a homotopy cartesian diagram
\[
\xymatrix{
\Riem^+ (W)_g \hq \Diff_\partial (W) \ar[r] \ar[d] & T_{\infty}^+ \ar[d]^{p}\\
B \Diff_\partial (W) \ar[r]^-{\alpha_W} & \loopinf{}_0 \MT{\theta}.
}
\]
\end{MainThm}

\begin{remark}
We shall show in \cite{ERWpsc3} that this theorem continues to hold under the weaker assumption that $(W, \partial W)$ is $2$-connected. This also follows from work of Perlmutter \cite{PerlmutterPSC}.
\end{remark}

If $W$ is spin with $\pi_1 (W)=G$ then $\MT{\theta} = \MT{Spin (2n)} \wedge BG_+$. Taking the fibre transport of the fibration $p$ at a specific basepoint yields the map $\Psi$ in Theorem \ref{main-theorem-mapintoRiem}. As in \cite{BERW} the construction of the fibration $T^+_\infty \to \loopinf{}_0 \MT{\theta}$ is by obstruction-theory, and it is here that Theorems \ref{thm:existence-stable-metrics1} and \ref{thm:right-and-left-stable} are used.

The second part of Theorem \ref{main-theorem-mapintoRiem} is index-theoretic. The main difference to \cite{BERW} is that the secondary index invariant takes place in the $K$-theory of $\cstarred(G)$ and not of $\bR$, and uses elliptic operators with coefficients in $\cstarred(G)$. The necessary elliptic regularity theory is developed in \cite{JEIndex1}. With these results, the index-theoretic part of the argument is largely the same as in \cite{BERW}, with the exception of the analogue of the Atiyah--Singer index theorem. In Section \ref{sec:indextheory}, we review the necessary changes. 

At this point, we have established Theorem \ref{main-theorem-mapintoRiem} for certain $W$ (those which are $(n-1)$-connected relative to their boundary) and certain boundary conditions (those which extend to a stable metric on $W$). To extend this to general manifolds and boundary conditions, we need the additivity theorem for the index and an embedding trick; stable metrics are used again here.

\vspace{1.5ex}

\noindent\textbf{Acknowledgements.} J.\ Ebert wishes to acknowledge a number of useful conversations with several colleagues, among them Michael Joachim, Wolfgang L\"uck and Rudolf Zeidler about the Baum--Connes conjecture and higher index theory. The authors were partially supported by the SFB 878 and by EPSRC grant EP/M027783/1.

\section{Spaces of psc metrics on manifolds with boundaries and corners}\label{sec:psc-spaces-on-mfds-w-corners}

\subsection{Spaces of psc metrics on manifolds with boundaries}

For a closed manifold $M$, we let $\Riem (M)$ be the space of all Riemannian metrics, equipped with the usual Fr\'echet topology and we let $\Riem^+ (M) \subset \Riem (M)$ be the open subspace of all Riemannian metrics with positive scalar curvature. 
Let $W$ be a compact manifold with boundary $M$. We assume that the boundary of $W$ comes equipped with a collar $c: M \times [0,1] \to W$. The collar identifies $M \times [0,1]$ with a subset of $W$ and we usually use this identification without further mention.

For $1\geq \eps>0$, we denote by $\Riem^+ (W)^{\eps}$ the space of all Riemannian metrics $h$ on $W$ with positive scalar curvature such that $c^* h = g+dt^2$ on $M \times [0,\eps]$ for some metric $g$ on $M$. It is topologised as a subspace of the space of smooth symmetric $(2,0)$-tensor fields, with the usual Fr\'echet topology. 

If the scalar curvature of $h$ is positive, then $g \in \Riem^+ (M)$, and assigning to $h$ its boundary value $g$ defines a map 
\[
\res_\eps: \Riem^+ (W)^{\eps} \lra \Riem^+ (M) 
\]
which is continuous. We define $\Riem^+ (W)_g^\eps := \res^{-1}_\eps (g) \subset \Riem^+ (W)^\eps$, the space of all psc metrics on $W$ which on $M \times [0,\eps]$ are of the form $g +dt^2$. 

For $1 \geq \delta >\eps$, the inclusion $\Riem^+ (W)_g^\delta \to \Riem^+ (W)_g^\eps$ is a closed embedding, and it is a homotopy equivalence by \cite[Lemma 2.1]{BERW}. As in \cite{BERW}, when the collar length does not play a role we abbreviate
$$\Riem^+(W) := \Riem^+ (W)^\eps \quad\text{ and }\quad \Riem^+ (W)_g := \Riem^+ (W)_g^\eps$$
for implicitly fixed values of $\epsilon$.

\subsection{Smooth manifolds with corners}

The proof of Theorem \ref{thm:existence-stable-metrics1} requires to come to grips with spaces of psc metrics on \emph{manifolds with corners} (whereas Theorem \ref{thm:right-and-left-stable} is a more formal consequence of Theorem \ref{thm:existence-stable-metrics1}). In the following section, we describe in detail what we mean by this and prove a corner rounding result, which is a new key tool.

\begin{definition}\label{defn:Corners}
A \emph{smooth $d$-manifold with acute corners} $W$ is a (second countable, Hausdorff) topological space locally modelled on $ \bR^{d-2} \times[0,\infty)^2$ and the local diffeomorphisms which preserve the sets $  \bR^{d-2}\times\{0\}$, $  \bR^{d-2}\times [0,\infty) \times \{0\}$, and $ \bR^{d-2} \times\{0\} \times [0,\infty)$. 

The (topological) boundary $\partial W$ is decomposed into two codimension zero pieces $M_0^{d-1}$ corresponding to $ \bR^{d-2} \times[0,\infty) \times \{0\}$ in local coordinates, and $M_1^{d-1}$ corresponding to $ \bR^{d-2}\times \{0\} \times [0,\infty)$ in local coordinates, which intersect along a closed submanifold $M_{01}^{d-2}$ corresponding to $  \bR^{d-2}\times\{0\}$ in local coordinates, which is the common boundary of $M_0$ and $M_1$. 

A \emph{smooth d-manifold with obtuse corners} is analogous but locally modelled on $ \bR^{d-2} \times(\bR^2 \setminus (0,\infty)^2)$.
\end{definition}

We write the following constructions and results for manifolds with acute corners, but the case of obtuse corners can be treated in the same way with only notational differences.

Choose collars $b:M_{01} \times [0,1]\times \{0\} \to M_0$ and $c:M_{01} \times \{0\} \times [0,1] \to M_1$, so that $b(x,0,0)=c(x,0,0) =x$ for all $x \in M_{01}$. The manifolds $M_0 \cup_{M_{01}} (M_{01} \times (-\infty,0])$ and $M_1 \cup_{M_{01}} (M_{01} \times (-\infty,0])$, where we use the collars $b$ and $c$ to define their smooth structures, are called the elongations of $M_0$ and $M_1$. The collars extend to embeddings
\begin{align*}
b' : M_{01} \times (-\infty,1] \times \{0\} &\lra M_0 \cup_{M_{01}} (M_{01} \times (-\infty,0])\\
c':M_{01} \times \{0\} \times (-\infty,1] &\lra M_1 \cup_{M_{01}} (M_{01} \times (-\infty,0])
\end{align*}
For each pair $-\infty \leq u \leq 0 \leq v \leq 1$ we may form the pushout
\[
\xymatrix{
M_{01} \times [u,v]^2 \ar[r]^-{i_{M_0}} \ar[d]^{i_{M_1}} & (M_0 \cup_{M_{01}} M_{01} \times [u,0]) \times [u,v] \ar[d]^{j_{M_0}}\\
(M_1 \cup_{M_{01}} M_{01} \times [u,0]) \times [u,v] \ar[r]^-{j_{M_1}} & K^{[u,v]},
}
\]
where $i_{M_0} (x,s,t):= (b'(x,t),s)$, $i_{M_1} (x,s,t):= ( c' (x,s), t)$. 

\begin{defn}
A \emph{bicollar} of $W$ is an embedding $e:K^{[0,1]} \to W$ such that $e (j_{M_0} (x,0))=x$ for all $x \in M_0$ and $e(j_{M_1} (y,0))=y$ for all $y \in M_1$. The smooth manifold $\widetilde{W} := W \cup_{K^{[0,1]}} K^{(-\infty,1]}$ is called the \emph{elongation of $W$}.
\end{defn}

\begin{figure}[htb!]
\begin{center}
\includegraphics{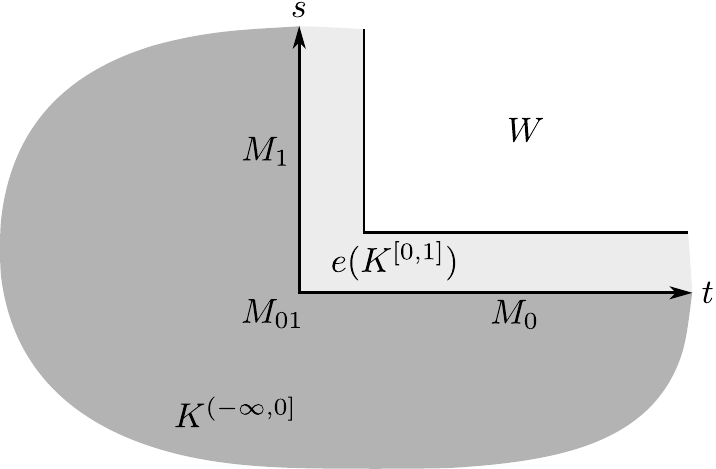}
\caption{The boundary structure of a manifold with (acute) corners}\label{fig:bicollars}
\end{center}
\end{figure}

There is a completely analogous model for an obtuse corner, which we shall omit.

\subsection{Rounding corners of smooth manifolds}

We first need to introduce some notation. 

\begin{defn}
A \emph{special curve in $\bR^2$} is a pair $(B,b)$, with $b \in [0,\infty)$ and $B \subset \bR^2$ a topological submanifold homeomorphic to $\bR$, such that 
\begin{enumerate}[(i)]
\item $B \cap ([0,\infty) \times \bR) =[0,\infty) \times \{-b \}$, and
\item $B \cap (\bR \times [0, \infty)) = \{-b\}\times [0, \infty)$. 
\end{enumerate}
We say that $(B_0,b_0) <(B_1,b_1)$ for two special curves if $B_0 \cap B_1 = \emptyset$ and $b_0<b_1$.
In that case, we denote by $[B_0,B_1]$ the region that lies between these two curves.
\end{defn}

\begin{figure}[htb!]
\begin{center}
\includegraphics{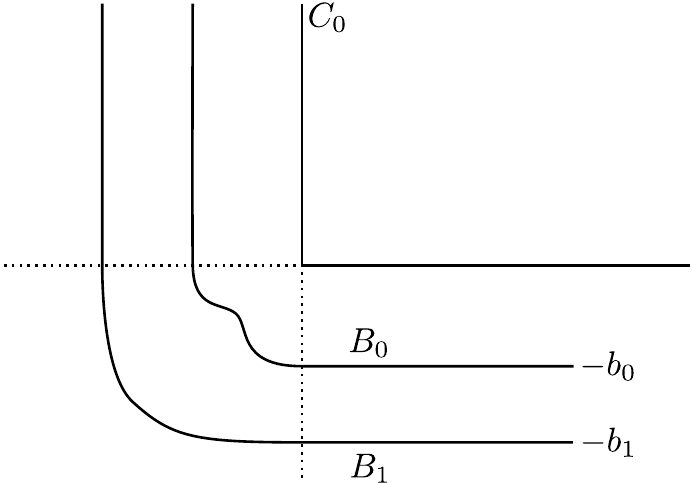}
\caption{Some special curves}\label{fig:specialcurve}
\end{center}
\end{figure}

If $(B,b)$ is a special curve then the number $b$ is determined by the curve $B$, so we abuse notation slightly by just writing $B$ for a special curve. For $b \in [0,\infty)$ we let
$$C_b := \{-b\} \times [-b, \infty) \cup  [-b, \infty)  \times \{-b\} \subset \bR^2.$$
Then $(C_b,b)$ is a special curve, and the only special curve with $b=0$ is $C_0$.
The part of $[B_0,B_1]$ outside the third quadrant consists of two strips. If $(B_0,b_0)< (B_1,b_1)$, we write
$$V_{B_0,B_1} := (M_{01} \times [B_0,B_1]) \cup (M_0 \times [-b_1,-b_0]) \cup (M_1 \times [-b_1,-b_0]) \subset K^{(-\infty,1]},$$
and for each special curve $B$ and bicollared manifold with acute corners $W$, we let
$$W_B:= W \cup_{\partial W} V_{C_0,B} \subset \widetilde{W}, $$
which is a smooth codimension zero submanifold with boundary of the elongation $\widetilde{W}$ if $B$ is smooth.

\begin{defn}
If $B$ is a smooth special curve, then $W_B$ is the \emph{result of rounding the corners of $W$}.
\end{defn}

The smooth manifold with boundary $W_B$ is independent of $B$ up to diffeomorphism, because the region $[C_0, B] \subset \bR^2$ is independent of $B$ up to diffeomorphism relative to $C_0$. Later we shall have to make more specific choices of $B$. If $a>0$ is chosen large enough so that $C_a > B$, then 
\[
W_B \cup_{\partial W} V_{B,C_a} \subset \widetilde{W}
\]
is a manifold with corners and it is diffeomorphic as such to $W$. In this way, we recover the manifold $W$ from $W_B$. 

\subsection{Collars}

We shall consider psc metrics on the manifolds $V_{B_0, B_1}$, so must clarify how we shall treat boundary conditions on such manifolds. If $B_0 = C_a$, then the manifold $V_{B_0, B_1}$ has an obtuse corner in the boundary $B_0$, and this has an evident bicollar; similarly, if $B_1 = C_a$  then the manifold $V_{B_0, B_1}$ has an acute corner in the boundary $B_1$, and this again has an evident bicollar. If instead $B$ is a smooth special curve then
$$(M_{01} \times B) \cup (M_0 \times \{-b\}) \cup (M_1 \times \{-b\}) \subset K^{(-\infty,1]}$$
is a codimension 1 submanifold and a choice of two-sided collar of $B$ (agreeing with the evident collar of each $\{-b\} \times [0,\infty)$ and $[-b,\infty) \times \{-b\}$) determines a two-sided collar of this submanifold.

\subsection{Spaces of psc metrics on manifolds with corners}

We will now define spaces of metrics of positive scalar curvature on a manifold with corners $W$ equipped with a bicollar $e : K^{[0,1]} \to W$. To ease notation, we will not mention the embedding $e$ of the bicollar into $W$ and pretend that $K^{[0,1]} \subset W$.

\begin{defn}
Let $g_{01} \in \Riem^+ (M_{01})$, $0<\eps<1$, $g_0 \in \Riem^+ (M_0)^{\eps}_{g_{01}}$ and $g_1 \in \Riem^+ (M_1)^{\eps}_{g_{01}}$, and write $g_{\partial W}:= (g_0 , g_1) \in \Riem^+ (M_0)^\eps \times_{\Riem^+ (M_{01})} \Riem^+ (M_1)^\eps$. Let $g_K$ be the Riemannian metric on $K^{(-\infty,\eps]}$ defined by
\[
g_K=
\begin{cases}
 g_{01} + dt^2+ ds^2 &\text{ on }K^{(-\infty,\eps]},\\
 g_0 + ds^2 & \text{ on } M_0 \times (-\infty,\eps],\\
g_1 + dt^2 & \text{ on }M_1 \times (-\infty,\eps].
\end{cases}
\]
The space of all psc metrics $h$ on $W$ such that $h|_{K^{[0,\eps]}} = g_K$ is denoted $\Riem^+  (W)^{\epsilon}_{g_{\partial W}}$.
\end{defn}

The following collar stretching lemma is the analogue of \cite[Lemma 2.1]{BERW}. In particular it implies that the homotopy type of $\Riem^+(W)^{\epsilon}_{g_{\partial W}}$ is independent of $\epsilon$, which allows us to neglect this from the notation when it is not important.

\begin{lem}\label{lem:BicollarLength}
If $0 < \eps' < \eps$ then the inclusion map $\Riem^+(W)^{\epsilon}_{g_{\partial W}} \to \Riem^+(W)^{\epsilon'}_{g_{\partial W}}$ is a homotopy equivalence.
\end{lem}
\begin{proof}
By breaking up this inclusion as a sequence of inclusions
$$\Riem^+(W)^{\epsilon}_{g_{\partial W}} \lra \Riem^+(W)^{\epsilon_1}_{g_{\partial W}} \lra \Riem^+(W)^{\epsilon_2}_{g_{\partial W}} \lra \cdots \lra  \Riem^+(W)^{\epsilon'}_{g_{\partial W}},$$
without loss of generality we may suppose that $\epsilon-\epsilon' < 1-\epsilon$ and that $2\epsilon' > \epsilon$.
Choose $0 < \delta < 2\epsilon'-\eps$ and a diffeomorphism $h_1 : [0,1] \to [0,1]$ such that
\begin{enumerate}[(i)]
\item $h_1(t) \leq t$,
\item $h_1 (\eps)= \eps'$,
\item $h_1(0)=0$ and $h_1(1)=1$,
\item $h_1'\equiv 1$ near $[0,\delta]\cup \{1\} \cup [\eps', \eps+(\eps-\eps')]$.
\end{enumerate}
The only condition which it is not obvious to satisfy is that $h_1'\equiv 1$ near $[\eps', \eps+(\eps-\eps')]$, which is where the conditions above enter in order to guarantee that
$$h_1(\epsilon + (\epsilon-\eps')) = h_1(\epsilon) + (\epsilon-\eps') = \eps' + (\epsilon-\eps') = \epsilon < 1$$
and
$$h_1(\epsilon') = h_1(\epsilon - (\epsilon-\eps')) = \eps' - (\epsilon-\eps') = 2\eps' -\eps > \delta.$$
Now let 
$$h_u(t) = u \cdot h_1(t) + (1-u) \cdot t$$
be the linear interpolation, an isotopy from the identity map to $h_1$. This induces an isotopy of $K^{[0,1]}$, by taking products, and hence an isotopy $H_u$ of $W$ supported inside $e(K^{[0,1]})$.

Let $f_u : [0,1] \to [0,\infty)$ be the function
\[
 f_u (t):= \frac{1}{(h_u' (h_u^{-1}(t)))^2}
\]
(note that $f_u \equiv 1$ near $[0,\delta]\cup [\eps', \eps]\cup\{1\}$ for all $u$). Then $h_u^* (f_u (t) dt^2) = dt^2$.

\begin{figure}[htb!]
\begin{center}
\includegraphics{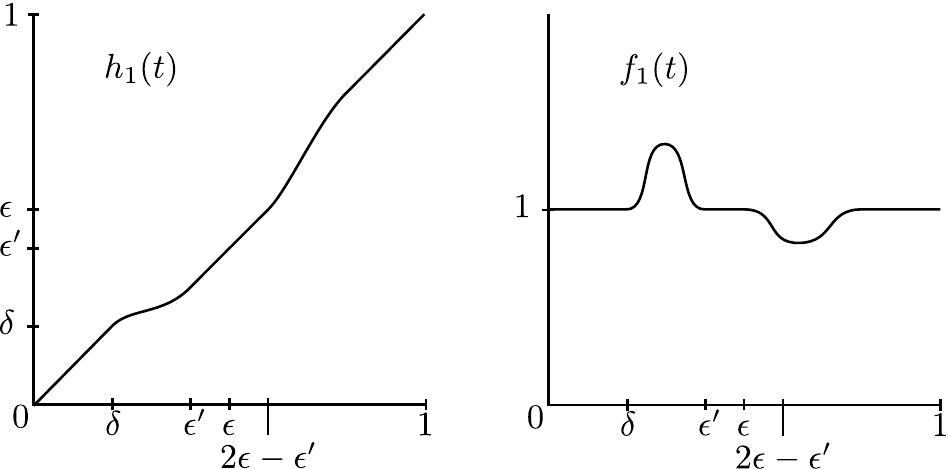}
\caption{The function $h_1(t)$ and its associated function $f_1(t)$.}\label{fig:functions}
\end{center}
\end{figure}

For a $h \in \Riem^+(W)^{\epsilon'}_{g_{\partial W}}$, define a Riemannian metric $h^u$ on $W$ by 
\[
 h^u := 
 \begin{cases}
  g_{01} + f_u (s)ds^2 + f_u (t)dt^2 & \text{on } M_{01} \times [0,\eps'] \times [0,\eps'] \\
  g_0  + f_u (s)ds^2 & \text{on } (M_0 \setminus b(M_{01} \times [0,\eps'] \times \{0\})) \times [0,\eps']\\
  g_1  + f_u (t)dt^2 & \text{on } (M_1 \setminus c(M_{01} \times \{0\} \times [0,\eps'])) \times [0,\eps']\\  
  h & \text{elsewhere}
 \end{cases}
\]
(this is not yet the homotopy we wish to construct).
This is a smooth Riemannian metric, as the $g_i$ are $\eps$-collared, $h$ is $\eps'$-bicollared, and $f_u \equiv 1$ near $\eps'$. Furthermore, it depends continuously on $u \in [0,1]$, is $\delta$-bicollared, and (as $f_u (t)dt^2$ is flat) it has positive scalar curvature.

Now we define a homotopy $F: \Riem^+(W)^{\epsilon'}_{g_{\partial W}} \times [0,1] \to \Riem^+(W)^{\epsilon'}_{g_{\partial W}}$ by the formula $F(h,u) = H_u^* h^u$, which satisfies
\[
 H_u^*h^u = 
 \begin{cases}
  g_{01} + ds^2 + dt^2 & \text{on } M_{01} \times [0,h_u^{-1}(\epsilon')] \times [0,h_u^{-1}(\epsilon')] \\
  g_0 + ds^2 & \text{on } (M_0 \setminus b(M_{01} \times [0,h_u^{-1}(\epsilon')] \times \{0\})) \times [0,h_u^{-1}(\epsilon')]\\
  g_1 + dt^2 & \text{on } (M_1 \setminus c(M_{01} \times \{0\} \times [0,h_u^{-1}(\epsilon')])) \times [0,h_u^{-1}(\epsilon')]\\  
  H_u^*h & \text{elsewhere}.
 \end{cases}
\]
This is a smooth Riemannian metric, is $h_u^{-1}(\epsilon')$-bicollared,  depends continuously on $u \in [0,1]$, and has positive scalar curvature.  As $F(h,0)=h$ and $F(h,1)$ is $\epsilon$-bicollared, this gives a homotopy ending in the subspace $\Riem^+(W)^{\epsilon}_{g_{\partial W}} \subset \Riem^+(W)^{\epsilon'}_{g_{\partial W}}$. Furthermore, as $f_u \equiv 1$ near $[\epsilon', \epsilon]$, if $h$ is $\eps$-bicollared then so is $ H_u^*h^u$ for all $u$, so this homotopy preserves the subspace $\Riem^+(W)^{\epsilon}_{g_{\partial W}}$. Thus the homotopy $F$ is a weak deformation retraction of $\Riem^+(W)^{\epsilon'}_{g_{\partial W}}$ into $\Riem^+(W)^{\epsilon}_{g_{\partial W}}$.
\end{proof}

There are several consequences of this lemma which we shall use more often than the result itself. Firstly, we can interpret adding an external collar as shrinking collar length.

\begin{cor}\label{cor:BicollarStretching}
The map
$$ S:= \_ \cup (V_{C_a, C_0}, g_K\vert_{V_{C_a, C_0}}) : \Riem^+ (W)^\eps_{g_{\partial W}} \lra \Riem^+ (W_{C_a})^\eps_{g_{\partial W_{C_a}}}$$
is a homotopy equivalence.
\end{cor}
\begin{proof}
The source may be identified with $\Riem^+ (W_{C_a})^{a+\eps}_{g_{\partial W_{C_a}}}$, whereupon this map becomes the natural inclusion $\Riem^+ (W_{C_a})^{a+\eps}_{g_{\partial W_{C_a}}} \subset \Riem^+ (W_{C_a})^{\eps}_{g_{\partial W_{C_a}}}$ which is a homotopy equivalence by Lemma \ref{lem:BicollarLength}.
\end{proof}

Secondly, we can add an external collar only to one part of the boundary, say $M_0$. Precisely, we can form the manifold $W \cup_{M_0} (M_0 \times [-a,0])$ whose boundary decomposes as $M_0$ and $M_1 \cup_{M_{01}} (M_{01} \times [-a,0])$. We have a boundary condition
$$\tilde{g}_{\partial W} := (g_0,g_1 \cup (g_{01} + ds^2)) \in \Riem^+(M_0)^{\epsilon}_{g_{01}} \times_{\Riem^+(M_{01})} \Riem^+(M_1 \cup_{M_{01}} (M_{01} \times [-a,0])^{a+\eps}_{g_{01}}$$
for this bicollared manifold.

\begin{cor}\label{cor:to-lem:BicollarStretching}
With the notations introduced above, the gluing map
\[
\Riem^+ (W)_{g_{\partial W}}  \lra \Riem^+ (W \cup_{M_0} M_0 \times [-a,0])_{\tilde{g}_{\partial W}}
\]
which glues in the psc metric $g_0 + ds^2$ is a homotopy equivalence.
\end{cor}

\begin{proof}
By construction, the map $S$ from Corollary \ref{cor:BicollarStretching} factors as 
\[
 \Riem^+ (W)_{g_{\partial W}}  \lra \Riem^+ (W \cup_{M_0} M_0 \times [-a,0])_{\tilde{g}_{\partial W}} \lra \Riem^+ (W_{C_a})_{g_{\partial W_{C_a}}},
\]
and both maps glue in a cylinder along one of the parts of the boundary. We have to prove that the first of those maps is a weak homotopy equivalence, but the second one is of the same type. The conclusion now is as follows: since the composition is a homotopy equivalence, the first map is split monomorphic in the homotopy category, and the second is split epimorphic. Since the second map is of the same type as the first one, it is also split monomorphic, hence a homotopy equivalence. It follows that the first map is a homotopy equivalence as well.
\end{proof}

\subsection{The corner rounding theorem}

The main result of this section is as follows. We adopt our convention of omitting the bicollar length.

\begin{thm}\label{thm:corner-smoothing}
Let $W$ be as above, $g_{\partial W}=(g_0,g_1) \in \Riem^+ (M_0) \times_{\Riem^+ (M_{01})} \Riem^+ (M_1)$. Then there exists 
\begin{enumerate}[(i)]
\item a collared smooth special curve $(B, b)$, and a psc metric $h'$ on $V_{C_0, B}$ which is equal to $g_{\partial W}$ over $C_0$ and is equal to some $g_{\partial W_B}$ over $B$,

\item a $b'>0$ such that $C_{b'} > B$, and a psc metric $h''$ on $V_{B,C_{b'}}$ which is equal to $g_{\partial W_B}$ over $B$ and is equal to $g_{\partial W_{b'}}$ over $C_{b'}$,
\end{enumerate}
such that the maps
$$\Riem^+ (W)_{g_{\partial W}} \overset{\_ \cup (V_{B, C_0}, h')}\lra \Riem^+ (W_{B})_{g_{\partial W_B}} \overset{\_ \cup (V_{C_{b'}, B}, h'')}\lra \Riem^+ (W_{C_{b'}})_{g_{\partial W_{C_{b'}}}}$$
are weak homotopy equivalences, and the composition is homotopic to gluing on $(V_{C_{b'}, C_0}, g_K\vert_{V_{C_{b'}, C_0}})$.

Furthermore, we may suppose that $h'$ and $h''$ are equal to $g_0 + ds^2$ near $M_0 \times (-\infty,1]$ and to $g_1 + dt^2$ near $M_1 \times (-\infty,1]$.
\end{thm}

\begin{proof}
Let $\varphi: \bR \times (0,4) \to \bR^2$ be a smooth embedding such that
\begin{enumerate}[(i)]
\item $\varphi (x,t)= (-t,-1-x)$ for $x \leq -1$,
\item $\varphi (x,t)= (x-1,-t)$ for $x \geq 1$,
\item $\varphi (\bR \times (0,4)) = ((-4,\infty) \times (-4,0)) \cup ((-4,0) \times (-4,\infty))$, 
\item $C_2 \subset \varphi (\bR \times (1,3))$.
\end{enumerate}
Define special curves
\[
B_1:= \varphi (\bR \times \{ 1\}) \quad \text{ and } \quad B_3:=  \varphi (\bR \times \{ 3\}) \subset \bR^2.
\]
Note that $C_0< B_1 < C_2 < B_3$, and that $\varphi$ restricts to a diffeomorphism from $\bR \times [1,3]$ to $[B_1,B_3]$. 

\begin{figure}[htb!]
\begin{center}
\includegraphics{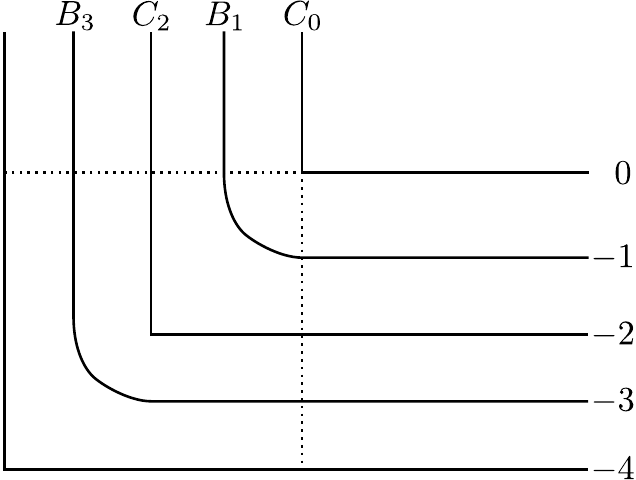}
\caption{Special curves used in the proof of Theorem \ref{thm:corner-smoothing}}\label{fig:specialcurves-proof-cornersmooth}
\end{center}
\end{figure}

Next, choose two isotopies of Riemannian metrics on $\bR^2$, $m_u$ and $n_u$, $u \in [0,1]$, such that
\begin{enumerate}[(i)]
\item $m_0=ds^2 + dt^2$ is the flat metric.
\item For each $u \in [0,1]$, $m_u$ coincides with $ds^2+dt^2$ outside $[-2,0]^2$.
\item For some fixed $\delta>0$, $\varphi^* m_1$ is the flat metric on $ \bR \times [1-\delta,1+\delta]$.
\item $n_0$ agrees with $m_1$ on a neighbourhood of $[B_1,C_2]$.
\item For each $u \in [0,1]$, $n_u$ coincides with $ds^2+dt^2$ outside $[-4,0]^2$.
\item $\varphi^* n_1$ is the flat metric in a neighbourhood of $\bR \times [3,1]$.
\end{enumerate}

Of course, we cannot require that $m_u$ or $n_u$ have positive or even nonnegative scalar curvature. But because both families of metrics are assumed to be constant (and flat) outside a compact subset of the plane, there exists a $\lambda>0$ such that
$$\scal (n_u), \scal(m_u) \geq -\lambda$$
for all $u \in [0,1]$. The following rescaling argument proves that we can make the constant $\lambda$ as small as we want, at the expense of stretching the region where the metrics $n_u$ and $m_u$ are not standard. More precisely, let $a>0$. Let $H_a: \bR^2 \to \bR^2$ be the homothety $H_a (x):= \frac{1}{a} x$. Then $a^2 H_{a}^{*}(ds^2 + dt^2) = ds^2 + dt^2$. Let 
\[
m_u^a:= a^2 H_{a}^{*}m_u \quad \text{ and } \quad n_u^a:= a^2 H_{a}^{*}n_u.
\]
By the well-known scaling identity for the curvature \cite[p. 136]{GHL}, we have
\[
\scal (n_u^a), \scal(m_u^a) \geq -\frac{1}{a^2}\lambda.
\]
These rescaled metrics $m_u^a$ and $n_u^a$ have the same properties as those listed above, but with $B_1$ replaced by $B_a$, $C_2$ by $C_{2a}$ and so on (one also has to rescale the embedding $\varphi$).
Let
\[
\kappa:= \inf \{\scal_{g_{01}}, \scal_{g_0}, \scal_{g_1} \}>0
\]
and choose $a \geq \sqrt{\frac{2 \lambda}{\kappa}}$. Define $h_u \in \Riem^+ (V_{C_0,C_{2a}})_{g_{\partial W}, g_{\partial W}}$ for $u \in [0,1]$, by
\[
h_u = 
\begin{cases}
g_0 + ds^2 & \text{on $M_0 \times [-2a,0]$}\\
g_1 + dt^2 & \text{on $M_1 \times [-2a,0]$}\\
g_{01} + m_u^a & \text{on $M_{01} \times [-2a,0]^2$}.
\end{cases}
\]
Over $M_{01} \times [-2a,0]^2$, this is a product metric, and (using \cite[\S 3.15]{GHL}) we have
$$ \scal(g_{01} + m_u^a) =  \scal(g_{01}) + \scal(m_u^a) \geq \kappa -\frac{1}{a^2}\lambda \geq \frac{1}{2} \kappa.$$
Over $M_0 \times [-2a,0]$ and $M_1 \times [-2a,0]$, $\scal(h_u) \geq \kappa$, so that $h_u$ is a psc metric. 

Gluing on $V_{C_0,C_{2a}}$ with the metric $h_u$ defines a homotopy of maps
$$\mu_{h_u}:\Riem^+ (W)_{g_{\partial W}} \lra \Riem^+ (W_{C_{2a}})_{g_{\partial W_{C_{2a}}}}.$$
The map $\mu_{h_0}$ is a weak homotopy equivalence by Corollary \ref{cor:BicollarStretching}, since it is given by gluing on $V_{C_0, C_{2a}}$ equipped with the metric $m_0 = ds^2 + dt^2$. Hence $\mu_{h_u}$ is a weak homotopy equivalence for all $u\in [0,1]$.  By construction, the metric $h_1$ is collared near $B_a$ so can be split along it, and so is the union of two psc metrics $h' \in \Riem^+ (V_{C_0, B_a})_{g_{\partial W}, g_{\partial W_{B_a}}}$ (this \emph{defines} $g_{\partial W_{B_a}}$) and $h'' \in \Riem^+ (V_{B_a,C_{2a}})_{g_{\partial W_{B_a}}, g_{\partial W_{C_{2a}}}}$. Therefore, the weak equivalence $\mu_{h_1}$ factors as
$$\Riem^+ (W)_{g_{\partial W}} \stackrel{\mu_{h'} }{\lra} \Riem^+ (W_{B_a})_{g_{\partial W_{B_a}}} \stackrel{\mu_{h''}}{\lra} \Riem^+ (W_{C_{2a}})_{g_{\partial W_{C_{2a}}}}.$$

So far, this proves only that $\mu_{h'}$ has a left-inverse and $\mu_{h''}$ has a right-inverse (in the weak homotopy category). To show that $\mu_{h'}$ is a weak equivalence, it suffices to show that $\mu_{h''}$ has a left-inverse as well, and for this, we employ the family $n_u$ constructed above. We define a metric $h_u$, for $u \in [1,2]$, on $V_{B_a, B_{3a}}$ by the formula
\[
h_u =
\begin{cases}
g_0 + ds^2 & \text{on $M_0 \times [-3a,-a]$}\\
g_1 + dt^2 & \text{on $M_1 \times [-3a,-a]$}\\
g_{01} + n_{u-1}^a & \text{on $M_{01} \times [B_a, B_{3a}]$}.
\end{cases}
\]
By construction $h_u$ has positive scalar curvature, and furthermore the metric $h_1$ is collared near $B_{3a}$ so can be written as the union of the metric $h''$ and a metric $h''' \in \Riem^+ (V_{C_{2a}, B_{3a}})_{g_{\partial W_{C_{2a}}}, g_{\partial W_{B_{3a}}}}$. So gluing in $h_u$ defines a homotopy 
\[
\nu_{h_u}: \Riem^+ (W_{B_a})_{g_{\partial W_{B_a}}} \lra \Riem^+ (W_{B_{3a}})_{g_{\partial W_{B_{3a}}}}, u \in [1,2].
\]
We can write $\nu_{h_1} = \mu_{h'''} \circ \mu_{h''}$. Also by construction, $\nu_{h_2}$ glues in a cylinder metric and hence is a weak equivalence by \cite[Lemma 2.1]{BERW}. Thus $\nu_{h_1}$ is a weak equivalence, which completes the proof.
\end{proof}

\section{Stable metrics}\label{sec:StableMetrics}

\subsection{Summary of the section}

If $W: M_0 \leadsto M_1$ is a cobordism with collared boundaries and $g_i \in \Riem^+ (M_i)$, $i=0,1$, we let $\Riem^+ (W)_{g_0,g_1}:= \Riem^+ (W)_{g_0 \coprod g_1}$ be the space of all metrics of positive scalar curvature on $W$ which are equal to $g_i + dt^2$ near $M_i$, with respect to the given collars.
For any $h \in \Riem^+ (W)_{g_0,g_1}$, there are composition maps
\[
\mu (h, \_): \Riem^+ (V)_{g_1,g_2} \lra \Riem^+ (W \cup V)_{g_0,g_2}
\]
and 
\[
\mu ( \_,h): \Riem^+ (V')_{g_{-1},g_0} \lra \Riem^+ (V' \cup W)_{g_{-1},g_1}
\]
defined for all cobordisms $V: M_1 \leadsto M_2$, $V': M_{-1} \leadsto M_0$ and boundary conditions $g_{-1}\in \Riem^+ (M_{-1})$ and $g_2 \in \Riem^+ (M_2)$.

\begin{defn}\label{defn:right-stablemaintext}
Let $W: M_0 \leadsto M_1$ be a cobordism and let $h \in \Riem^+ (W)_{g_0,g_1}$. Then $h$ is called \emph{left-stable} if the map $\mu (\_, h): \Riem^+ (V)_{g_{-1},g_0} \to \Riem^+ (V \cup W)_{g_{-1},g_1}$ is a weak equivalence for all cobordisms $V:M_{-1}\leadsto M_0$ and all boundary conditions $g_{-1}$. Dually, $h$ is \emph{right-stable} if the map $\mu (h,\_): \Riem^+ (V)_{g_1,g_2} \to \Riem^+ (W \cup V)_{g_0,g_2}$ is a weak equivalence for all cobordisms $V: M_1 \leadsto M_2$ and all boundary conditions $g_2$. Finally, $h$ is \emph{stable} if it is both, left and right stable.
\end{defn}

To state the main result of this section, for a codimension $0$ open submanifold $C \subset \inter{M}$ which is the interior of a submanifold with boundary $\bar{C} \subset \inter{M}$, and a $g_C \in \cR^+(C)$, we let
$$\cR^+(M, C; g_C) := \{g \in \cR^+(M) \, \vert \, g\vert_C = g_C\}.$$
For a cobordism $W : M_0 \leadsto M_1$ and a codimension 0 open submanifold $[0,1] \times C \subset W$ such that $M_i \cap ([0,1] \times C) = \{i\} \times C$, given $g_i \in \cR^+(M_i, C ; g_C)$ we let
$$\cR^+(W, [0,1] \times C; dt^2 + g_C)_{g_0, g_1} := \{g \in \cR^+(W)_{g_0, g_1} \, \vert \, g\vert_{[0,1] \times C} = dt^2 + g_C\}.$$

\begin{thm}\label{thm:stableL}
Let $W: M_0 \leadsto M_1$ be a cobordism of dimension $\geq 6$.
\begin{enumerate}[(i)]
\item Let $[0,1] \times C \subset W$ be a codimension 0 submanifold as above, and let $g_C \in \cR^+(C)$. If $(W, M_1)$ is 2-connected, then for each $g_0 \in \Riem^+ (M_0, C;g_C)$, there is $g_1 \in \Riem^+ (M_1, C; g_C)$ and a right stable $h \in \Riem^+ (W, [0,1] \times C; dt^2 + g_C)_{g_0,g_1}$.
\item If both $(W, M_1)$ and $(W, M_0)$ are 2-connected, then every right stable $h \in \Riem^+ (W)_{g_0,g_1}$ is also left stable.
\end{enumerate}
\end{thm}

We shall only need the absolute ($C=\emptyset$) version in this paper (the full relative version will be used in \cite{ERWpsc3}). The proof is an elaboration of the Gromov--Lawson surgery technique, which we shall first review.

\subsection{The theorem of Gromov--Lawson and Chernysh}\label{subsec:gromov-lawson-recap}

\begin{defn}
By $g_\round^{k-1} \in \Riem (S^{k-1})$, we denote the round metric on $S^{k-1}$, i.e. the metric induced from euclidean metric by the standard inclusion $S^{k-1} \subset \bR^{k}$. It has constant scalar curvature $\scal (g_\round^{k-1}) = (k-1)(k-2)$ which is positive if $k \geq 3$. 

Let $\delta >0$. A \emph{$\delta$-torpedo metric} $g_\tor^k$ on $\bR^{k}$, $k \geq 3$, is an $O(k)$-invariant metric such that 
$\scal (g_\tor^k)\geq \frac{1}{\delta^2}$ and such that 
\[
\varphi^* g_\tor^k = dr^2 + \delta g_\round^{k-1}
\]
near $[1,\infty)\times S^{k-1}$ for some $\delta >0$, where $\varphi:(0,\infty) \times S^{k-1} \to \bR^k \setminus 0$ is the diffeomorphism defined by $(r,x)\mapsto rx$.
\end{defn}

In this work, we are mostly interested in the case $\delta =1$. We fix such a torpedo metric once and for all and refer to \cite[\S 2.3]{Walsh01} for more details.

\begin{defn}
Let $d-k \geq 3$, let $W^d$ be a compact manifold boundary $M$ and collar $[0,\infty) \times M \subset W$, let $V^k$ be a compact manifold with boundary $N$ and collar $[0,\infty) \times N \subset V$, and let $\phi: V \times \bR^{d-k} \to W$ be an embedding. Assume that $\phi^{-1} (M \times [0,\infty)) = N \times [0,\infty)$, and such that inside the collar $\phi$ is of the form $(x,t)\mapsto (\varphi (x),t)$ for some embedding $\varphi: N \times \bR^{d-k} \to M$. 

Let $h_V \in \Riem (V)_{g_V}$ be a metric and pick $\delta>0$ such that $\scal (h_V) + \delta >0$ and fix a $\delta$-torpedo metric $g_\tor^{d-k}$ on $\bR^{d-k}$. By 
\[
\Riem^+ (W,\phi)^\eps \subset \Riem^+ (W)^\eps,
\] 
we denote the space of all $\eps$-collared psc metrics $h$ on $W$ such that $\phi^* h = h_V + g_\tor^{d-k}$ near $V \times D^{d-k} \subset V \times \bR^{d-k}$. 
\end{defn}

For sufficiently small $\eps$, one can view $\Riem^+ (W,\phi)^{\eps}$ as a space of psc metrics on the manifold with corners $W \setminus \phi(V \times \mathring{D}^{d-k})$. Corollary \ref{cor:to-lem:BicollarStretching} allows us to neglect the notation $\eps$ from $\Riem^+ (W,\phi)^\eps$. In the cases of interest to us, $h_V$ has nonnegative scalar curvature, in which case we pick $\delta =1$. The following result due to Chernysh \cite[Theorem 1.3]{Chernysh2} is a sharpening of a classical result by Gromov and Lawson \cite{GL}, is of crucial importance for this paper:

\begin{thm}[Chernysh]\label{thm:chernysh-theorem}
Assume that $d-k \geq 3$. Then for each $g \in \Riem^+ (M,\varphi)$, the inclusion
\[
\Riem^+ (W,\phi)_g \lra \Riem^+ (W)_g
\]
is a weak equivalence.
\end{thm}

An immediate corollary of Theorem \ref{thm:chernysh-theorem} is the following cobordism invariance result.

\begin{defn}
Let $W$ be a compact $d$-dimensional manifold, possibly with boundary $M$. A surgery datum (i.e. embedding) $\phi:S^{k} \times \bR^{d-k} \to \inter{W}$ is \emph{admissible} if $2 \leq k \leq d-3$.
\end{defn}

We let 
$$W_\phi:= (W\setminus \phi(S^k \times D^{d-k})) \cup_{S^k \times S^{d-k-1} } D^{k+1} \times S^{d-k-1}$$
be the result of performing a surgery along $\phi$.

\begin{cor}\label{cor:chernysh-theorem}
An admissible surgery datum $\phi$ determines a preferred homotopy class of weak homotopy equivalences
\[
\SE_\phi: \Riem^+ (W)_g \simeq \Riem^+ (W_{\phi})_g,
\]
called the \emph{surgery equivalence} determined by $\phi$.
\end{cor}

We refer to \cite[Theorem 2.5]{BERW} for the (easy) derivation of this Corollary from Theorem \ref{thm:chernysh-theorem}.
We remark that there is no actual map between these spaces, but rather a zig-zag of weak equivalences. This is not relevant in the present paper: we only use $\SE_\phi$ to identify the sets of path components of both spaces.

\subsection{The stability condition}

In this subsection, we collect some fairly straightforward but important facts about stable metrics.
The following simple observation is immediate from the definitions and will be used repeatedly.

\begin{lem}\label{lem:2-out-of-three}
Let $(W, h): (M_0, g_0) \leadsto (M_1, g_1)$ and $(W', h'): (M_1, g_1) \leadsto (M_2, g_2)$ be psc cobordisms. Then
\begin{enumerate}[(i)]
\item If $(W,h)$ and $(W',h')$ are left-stable, then so is $(W \cup W', h\cup h')$. 
\item\label{it:2of3ii} If $(W,h)$ and $(W',h')$ are right-stable, then so is $(W \cup W', h\cup h')$. 
\item If $(W', h')$ and $(W \cup W', h\cup h')$ are left-stable, then so is $(W,h)$.
\item If $(W,h)$ and $(W \cup W', h\cup h')$ are right-stable, then so is $(W', h')$.
\end{enumerate}
\end{lem}

Note that the statements ``$(W', h')$ and $(W \cup W', h\cup h')$ right-stable $\Rightarrow$ $(W,h)$ right-stable'' and ``$(W', h')$ and $(W \cup W', h\cup h')$ left-stable $\Rightarrow$ $(W',h')$ left-stable'' do \emph{not} follow formally.

\begin{lem}\label{lem:isotopy-gives-stable-concordance}
Let $g_0, g_1 \in \Riem^+ (M)$ lie in the same path component. Then there is a stable metric $h \in \Riem^+ (M \times [0,1])_{g_0,g_1}$. 
\end{lem}

\begin{proof}
Let $g_t$, $t \in [0,1]$ be a smooth isotopy from $g_0$ to $g_1$. A result of Gajer \cite[p. 184]{Gajer} shows that there is $\Lambda >0$ such that whenever $f: \bR \to [0,1]$ is a smooth function with $|f'|, |f''| \leq \Lambda$, the metric $dt^2 + g_{f(t)}$ on $\bR \times M$ has positive scalar curvature. There exists $c>0$ and a smooth function $f: [0,2c] \to [0,1]$ such that $f\equiv 0$ near $0$ and $2c$, $b \equiv 1$ near $c$ and such that $|f'|, |f''| \leq \Lambda$. The formula 
\[
 h_s = dt^2 + g_{s f(t)}
\]
thus defines a curve in the space $\Riem^+ ([0,2c] \times M, [0,2c] \times L)_{g_0,g_0}$. But $h_0=dt^2 + g_0$ is stable by \cite[Corollary 2.2.]{BERW}, and so $h_1$ is a stable metric as well. We can write $h_1 = h' \cup h''$, with $h' \in \Riem^+ ([0,c] \times M, [0,c] \times L)_{g_0,g_1}$ and $h'' \in \Riem^+ ([c,2c] \times M, [c,2c] \times L)_{g_1,g_0}$. Then the map $\mu (\_, h')$ has a left inverse in the homotopy category, and $\mu(\_, h'')$ has a right inverse. Let $h'''= dt^2 + g_{a(t-2c)}$, a psc metric in $\Riem^+ ([2c,3c] \times M, [2c,3c] \times L)_{g_1,g_0}$. A similar homotopy proves that $h'' \cup h'''$ is stable. Therefore, the above map $\mu(\_, h'')$ also has a left-inverse in the homotopy category. Hence $\mu(\_, h'')$ is a weak equivalence, and so is $\mu(\_, h')$; in other words, $h'$ is left stable. An analogous argument shows that $h'$ is also right stable. Reparametrising the interval $[0,c]$ gives the desired stable $h \in \Riem^+ ([0,1] \times M, [0,1] \times L)_{g_0,g_1}$. 
\end{proof}

The next result shows the invariance of stable metrics under surgery equivalences. 

\begin{lem}\label{lem:stability-under-surgery}
Let $W: M_0 \leadsto M_1$ be a cobordism and $\phi: S^{k} \times D^{d-k} \to \inter{W}$ be an admissible surgery datum (i.e. $2 \leq k \leq d-3$). Let $[h]\in \pi_0\Riem^+ (W)_{g_0,g_1}$ and $[h'] \in  \pi_0\Riem^+ (W_{\phi})_{g_0,g_1}$ correspond under the weak equivalence $\SE_{\phi}$. Then $h'$ is left stable (right stable) iff $h$ is left stable (right stable).
\end{lem}

\begin{proof}
Exactly as the argument in the third paragraph of the proof of \cite[Theorem 2.6]{BERW}.
\end{proof}

\subsection{Construction of stable metrics}

We now prove Theorem \ref{thm:stableL}. The key step is the case of an elementary cobordism $W$. Let us introduce some notation. Throughout this section, $M_i$ will be a closed $(d-1)$-manifold and cobordisms between those are typically denoted $W: M_0 \leadsto M_1$. When $W$ is such a cobordism, we denote by $W^{op}: M_1 \leadsto M_0$ the same manifold, but viewed as a cobordism in the other direction. 

Let 
$$\varphi: S^{k-1} \times \bR^{d-k} \lra M_0$$
be a smooth embedding. We can view $\varphi|_{S^{k-1} \times D^{d-k}}$ as a surgery datum on $M_0$, as well as the attaching map for a $k$-handle. Let $M_1$ be the result of performing surgery on $\varphi$, and let $T_{\varphi}$ be the trace of the surgery. Then $T_{\varphi}: M_0 \leadsto M_1$ is a cobordism (an \emph{elementary cobordism of index $k$}). There is an embedding
\[
\phi^-:(D^k \times D^{d-k},S^{k-1} \times D^{d-k} ) \lra (T_{\varphi}, M_0)
\]
such that $\phi^- |_{S^{k-1} \times D^{d-k} }=\varphi$. The image of $\phi^-$ is denoted $H_{\varphi}^-$ and called the \emph{incoming handle}. Moreover, there is a dual embedding
\[
\phi^+:(D^{d-k}\times D^k, S^{d-k-1}\times D^k) \lra (T_{\varphi}, M_1)
\]
with image $H_{\varphi}^+$, called the \emph{outgoing handle}. 

It is important for us to consider the following decomposition of $T_\varphi$. For $r>1$, we let $D^n_r$ be the closed $n$-disc of radius $r$ and let $S^{n-1}_r$ be its boundary. Let 
\[
\varphi_0: S^{k-1} \times D^{d-k} \lra S^{k-1} \times D^{d-k}_2
\]
be the inclusion and let $U$ be the trace of a surgery on $\varphi_0$. This is a manifold with corners. The boundary of $U$ is decomposed into three pieces: $\partial_0 U:= S^{k-1} \times D^{d-k}_2$, $\partial_1 U$ is the result of a surgery along $\varphi_0$ and $\partial_2 U = S^{k-1} \times S^{d-k-1}_2 \times [0,1]$. Let $L := M_0 \setminus \varphi (S^{k-1} \times D^{d-k}_2)\subset M_0$. Then $T_\varphi = (L \times [0,1]) \cup_{S^{k-1} \times S^{d-k}_2 \times [0,1]} U$, see Figure \ref{fig:handle}. The incoming and the outgoing handle are disjoint from $L \times [0,1]$. 

\begin{figure}[htb!]
\begin{center}
\includegraphics{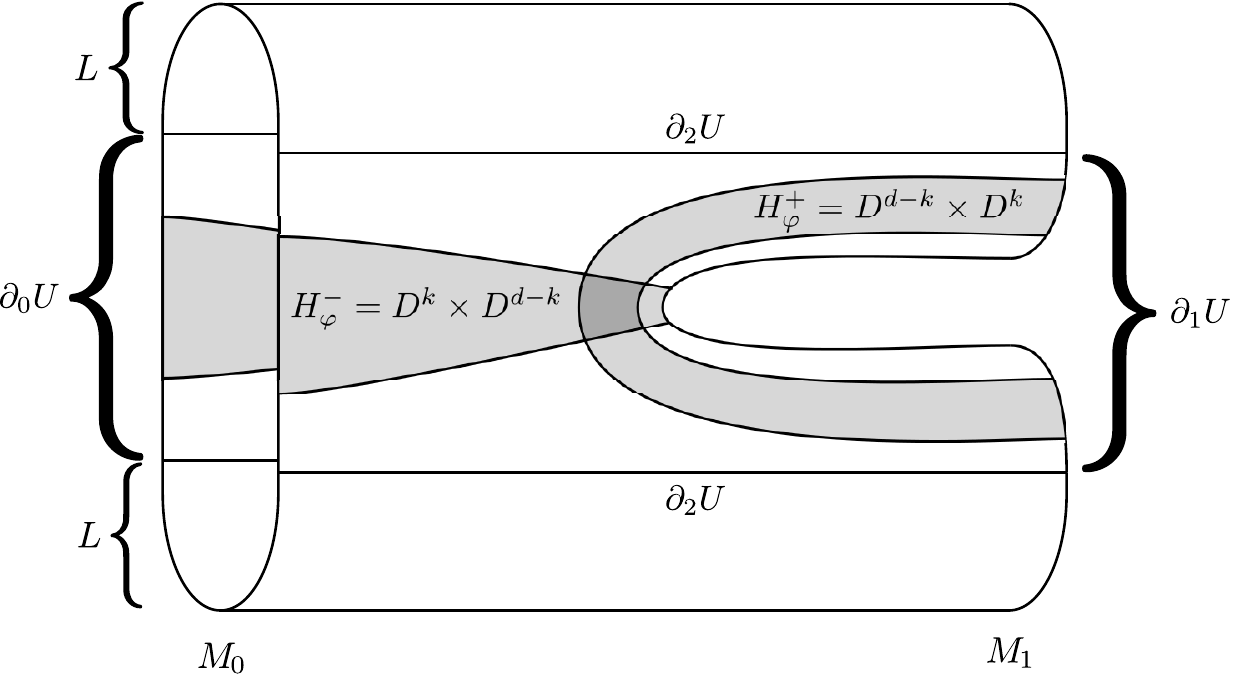}
\caption{Decomposition of an elementary cobordism}\label{fig:handle}
\end{center}
\end{figure}

Denote $\varphi^{op}:= \phi^+ |_{ S^{d-k-1}\times D^k}$. By surgery on $\varphi^{op}$, one recovers $M_0$ from $M_1$. Note furthermore that there is an identification $T_{\varphi^{op}} \cong (T_{\varphi})^{op}$, relative to the boundaries. Inside the composed cobordism $T_{\varphi} \cup T_{\varphi^{op}}$, we find an embedding $\phi = \phi^+ \cup \phi^-$ of $S^{d-k} \times D^k$. Doing surgery on $\phi$ results in a manifold diffeomorphic to the cylinder $M_0 \times [0,2]$ relative to the canonical diffeomorphism $\partial (T_{\varphi} \cup T_{\varphi^{op}}) \cong \partial(M_0 \times [0,2])$. 

\begin{lem}\label{lem:rightstablemetric-on-elementary-cobordism}
Let $k \leq d-3$. Then there is a psc metric $h$ on the manifold with corners $U$, such that 
\[
h = 
\begin{cases}
g_\round^{k-1} + g_\round^{d-k-1} + dt^2+ds^2 & \text{ near } \partial_2 U\\
g_\round^{k-1} + g_\tor^{d-k} + dt^2 & \text{ near } \partial_0 U
\end{cases}
\]
and such that the following property holds. If $M^{d-1}$ is a closed manifold, $g_0 \in \Riem^+ (M)$ and $\varphi : S^{k-1} \times \bR^{d-k}\to M$ an embedding such that $\varphi^* g_0 = g_\round^{k-1} + g_\tor^{d-k} $ near $\varphi (S^{k-1} \times D^{d-k}_2)$, then the psc metric $h \cup (g_0|_L + dt^2)$ on $U \cup (L \times [0,1])= T_\varphi$ is right-stable.
\end{lem}

\begin{proof}
We construct a right stable metric on $T_\varphi$ and justify during the course of the proof why it has the particular form stated in the lemma. The proof depends on the decomposition of $T_\varphi$ described above and shown in Figure \ref{fig:handle}.

Here we let $H^-_\varphi \subset U \subset W=T_{\varphi}$ be the image of the embedding $\phi^-:(D^k \times D^{d-k},S^{k-1} \times D^{d-k} ) \to (T_{\varphi}, M_0)$, and let $S= T_{\varphi} \setminus \inter{H^-_\varphi}$. Note that $S$ is a manifold with (acute) corners in the sense of Definition \ref{defn:Corners}, and that it is diffeomorphic to a manifold of the form $V_{B_1, C_2}$ considered in the proof of Theorem \ref{thm:corner-smoothing}. More precisely, the left boundary of $S$ is decomposed as 
\[
(M_0 \setminus \inter{S^{k-1} \times D^{d-k}}) \cup_{S^{k-1} \times S^{d-k-1} } (S^{k-1} \times D^k)
\]
and the right boundary of $S$ is the result of corner smoothing. Let 
\begin{align*}
g_{\partial} &= (g_0 |_{M_0 \setminus \inter{S^{k-1} \times D^{d-k}}}, g_{\round}^{k-1} + g_{\tor}^k) \\
&\,\,\,\,\in \Riem^+ (M_0 \setminus \inter{S^{k-1} \times D^{d-k}}) \times_{\Riem^+ (S^{k-1} \times S^{d-k-1})} \Riem^+  (S^{k-1} \times D^k). 
\end{align*}
We now \emph{define} $g_1 \in \Riem^+ (M_1)$ as $g_{\partial W_B}$, as explained in the proof of Theorem \ref{thm:corner-smoothing}.  By that theorem, there is a metric $h' \in \Riem^+ (S)_{g_{\partial},g_1}$ which is right stable in a slightly generalised sense: if $V: M_1 \leadsto M_2$ is a cobordism and $g_2 \in \Riem^+ (M_2)$, then $\Riem^+ (V)_{g_1,g_2} \to \Riem^+ (S \cup V)_{g_{\partial},g_2}$ is a weak equivalence. Moreover, by the last statement of Theorem \ref{thm:corner-smoothing}, we can assume that $h'$ is of the form $g_0|_L + dt^2 $ over $L \times [0,1]$, and $h'|_{S \setminus (L \times [0,1])}$ does not depend on the ambient manifold $M$. 

Let now $\Riem^+ (W,H^-_\varphi)_{g_0,g_1} \subset \Riem^+ (W)_{g_0,g_1}$ be the subspace of those psc metrics $h$ such that $(\phi^-)^* h = g_{\tor}^k + g_{\tor}^{d-k}$. By Theorem \ref{thm:chernysh-theorem} (iii), the inclusion 
\begin{equation*}
\Riem^+ (W,H^-_\varphi)_{g_0,g_1} \lra \Riem^+ (W)_{g_0,g_1}
\end{equation*}
is a weak homotopy equivalence because $d-k \geq 3$. The same is true when $W$ is replaced by $W \cup V$, for any other cobordism $V: M_1 \leadsto M_2$ and any boundary condition $g_2 \in \Riem^+ (M_2)$.
Now we consider the psc metric $h = ( g_{\tor}^k + g_{\tor}^{d-k}) \cup h'$ on $W = H^-_\varphi \cup S$. Let $V: M_1 \leadsto M_2$ be any other cobordism and $g_2 \in \Riem^+ (M_2)$. Then the map $(W, h) \cup -: \Riem^+ (V)_{g_1, g_2} \to \Riem^+ (W \cup V)_{g_0,g_2}$ is equal to the composition
\[
  \Riem^+ (V)_{g_1, g_2} \stackrel{(S, h') \cup -}{\lra} \Riem^+ (S \cup V)_{g_{\partial}, g_2} \cong \Riem^+ (W \cup V,H^-_\varphi)_{g_0,g_2} \lra \Riem^+ (W \cup V)_{g_0,g_2}
\]
of two weak homotopy equivalences, so $(W, h)$ is right stable. 
\end{proof}

\begin{proof}[Proof of Theorem \ref{thm:stableL} (i)]
By handle cancellation theory (see e.g.\ \cite{Kervaire, WallGC}), the cobordism $W$ admits a decomposition into elementary cobordisms of index $k \leq d-3$ (relative to $M_0$). Therefore, we may assume that $W=T_{\varphi}$, for some embedding $\varphi: S^{k-1} \times \bR^{d-k} \to M_0 \setminus C$ with $d-k \geq 3$. 

The next step is to turn the metric $g_0$ into some standard form. Because $d-k \geq 3$, there is a smooth family $g_t \in \Riem^+ (M_0)$ such that $\varphi\vert_{S^{k-1} \times D^{d-k}}^* g_1 = g_{\round}^{k-1} + g_{\tor}^{d-k}$, by Chernysh's theorem \cite[Theorem 1.1]{Chernysh} (the weaker result by Gajer \cite{Gajer} also suffices for this purpose), and this may be supported on a small neighbourhood of $\varphi(S^{k-1} \times D^{d-k})$ and hence disjointly from $C$, so gives a smooth family $g_t \in \Riem^+ (M_0, C)$. By Lemma \ref{lem:isotopy-gives-stable-concordance}, there is a stable metric $h_1 \in \Riem^+ ([0,1] \times M_0, [0,1] \times C; dt^2 + g_C )_{g_0,g_1}$ (this $h$ and $g_1$ are \emph{not} yet the metrics we are looking for). Lemma \ref{lem:rightstablemetric-on-elementary-cobordism} gives a right-stable metric $h_2 \in \cR^+(T_\varphi)_{g_1, g_2}$ which is cylindrical over $[0,1] \times L$ and so in particular over $[0,1] \times C$. The composition
$$h_1 \cup h_2 \in \cR^+([0,1] \times M_0 \cup_{M_0} T_\varphi)_{g_0, g_2} \cong\cR^+( T_\varphi)_{g_0, g_2}$$
is then right-stable by Lemma \ref{lem:2-out-of-three} (\ref{it:2of3ii}), and is cylindrical over $[0,1] \times C$ as required.
\end{proof}

\begin{proof}[Proof of Theorem \ref{thm:stableL} (ii)]
Let $W: M_0\leadsto M_1$ be a cobordism of dimension $d \geq 6$, assume that $M_i\to W$ are $2$-connected and let $h\in \Riem^+ (W)_{g_0,g_1}$ be a right stable psc metric on $W$. 
Using handle cancellation theory \cite{Kervaire, WallGC}, $W$ has a handle decomposition (relative to $M_0$) using handles of index $3 \leq k \leq d-2$ (in fact, if $d\geq 7$, one can assume that $3 \leq k \leq d-3$, but if $d=6$, this can fail, for example if $W$ is an $h$-cobordism with nontrivial Whitehead torsion). 

We have noted above that if $T_{\varphi}:M \leadsto M'$ is an elementary cobordism of index $k$, then the composite $T_{\varphi} \cup T_{\varphi}^{op}$ contains a surgery datum $\phi:S^{d-k} \times D^k \to T_{\varphi} \cup T_{\varphi}^{op}$ in its interior so that performing surgery on $\phi$ results in the cylinder $M\times [0,2]$. It follows that $W \cup W^{op}$ can be surgered to $M_0 \times [0,2]$, and this can be done using admissible surgeries, provided that all handles of $W$ have index $3 \leq k \leq d-2$, which is what we arranged in the first step of the proof. Hence there is a surgery equivalence
\[
\SE: \Riem^+ (M_0\times [0,2])_{g_0, g_0} \simeq \Riem^+ (W\cup W^{op})_{g_0,g_0}.
\]
Let $h'' \in \Riem^+ (W\cup W^{op})_{g_0,g_0}$ be in the component of the cylinder metric $g_0 + dt^2$ under $\SE$. By Lemma 
\ref{lem:stability-under-surgery} and \cite[Corollary 2.2]{BERW}, $h''$ is stable.

Now $h \in \Riem^+ (W)_{g_0,g_1}$ is right stable by assumption, so there is a psc metric $h' \in \Riem^+ (W^{op})_{g_1, g_0}$ such that $h\cup h'$ and $h''$ lie in the same path component of $\Riem^+ (W\cup W^{op})_{g_0,g_0}$ (just because $\mu (h,\_): \Riem^+ (W^{op})_{g_1, g_0} \to\Riem^+ (W\cup W^{op})_{g_0,g_0}$ is a weak homotopy equivalence). Since $h$ and $h \cup h'\simeq h''$ are right stable, so is $h'$, by Lemma \ref{lem:2-out-of-three}.

Let us summarise what we proved so far. Given a cobordism $W$ such that both inclusions $M_i \to W$ are $2$-connected, and given a right stable $h \in \Riem^+ (W)_{g_0,g_1}$, we find another right stable $h' \in \Riem^+ (W^{op})_{g_1,g_0}$ such that $h \cup h'$ is stable. 

The conclusion of the argument is formal. Namely, $\mu (\_,h\cup h')= \mu (\_,h' ) \circ  \mu (\_, h) $ is a weak equivalence, and so $\mu (\_, h)$ is ``split monomorphic in the weak homotopy category''. The argument leading to that conclusion only used that $h$ is a right stable psc metric on a cobordism whose inclusion maps $M_i \to W$ are both $2$-connected. Because $W^{op}$ and $h'$ also satisfy these conditions, we can apply the same argument to $h'$ and get that $\mu (\_, h')$ is split monomorphic in the weak homotopy category as well. Since $\mu (\_,h' ) \circ  \mu (\_, h) $ is a weak equivalence, $\mu (\_, h')$ is also split epimorphic. So $\mu (\_, h')$ is a weak equivalence, and so is $\mu (\_,h)$, which is what we had to prove.
\end{proof}

\section{Proof of the main factorization result}\label{sec:proof-constructivepart}

In this section we will prove Theorem \ref{thm:factorization}.

\subsection{A commutativity result}

Let $W^d$ be a compact manifold with boundary. The group $\Diff_{\partial} (W)$ of diffeomorphisms which are equal to the identity near $\partial W$ acts by pullback on the space $\Riem^+ (W)_g$ of psc metrics. This action defines a homomorphism $\Diff_{\partial} (W) \to \hAut (\Riem^+ (W)_g)$ to the monoid of homotopy automorphisms of\footnote{As in \cite[\S 4.1]{BERW}, one should replace $\Riem^+ (W)_g$ by a space having the homotopy type of a CW complex; we do so implicitly here.} $\Riem^+ (W)_g$. A key step of the paper \cite{BERW} (namely Theorem 4.1) is that under some circumstances the induced homomorphism $\pi_0 (\Diff_{\partial} (W)) \to \pi_0 \hAut (\Riem^+ (W)_g)$ has an abelian group as its image. This is to say: if $f_0,f_1 \in \Diff_{\partial} (W)$, then $f_0^* f_1^* $ and $f_1^* f_0^*: \Riem^+ (W)_g \to \Riem^+ (W)_g$ are homotopic. 
We introduce the following language.

\begin{defn}\label{defn:homotopy-abelian}
Let $X$ be a space and let $G$ be a topological group which acts on $X$. This gives rise to an $H$-space map $G \to \hAut (X)$. 
The action of $G$ on $X$ is called \emph{homotopy abelian} if the image of the induced map $\pi_0 (G) \to \pi_0 (\hAut(X))$ is an abelian group.
\end{defn}

\begin{thm}\label{thm:general-abelianness-thm}
Let $W^d:\emptyset \leadsto \partial W$ be a cobordism. Assume that $d \geq 6$ and that the inclusion map $\partial W \to W$ is $2$-connected. Assume that $g \in \Riem^+ (\partial W)$ is such that there exists a right stable $h \in \Riem^+ (W)_g$. Then the action of $\Diff_{\partial} (W) $ on $\Riem^+ (W)_{g}$ is homotopy abelian.
\end{thm}

\begin{remark}
Theorem 4.1 of \cite{BERW} applies under the following hypotheses: 
\begin{enumerate}[(i)]
\item $d \geq 5$,
\item $W$ is simply connected and spin,
\item $\partial W= S^{d-1}$ and $g$ is the round metric on $S^{d-1}$,
\item $W$ is spin cobordant to $D^d$, relative to its boundary. 
\end{enumerate}
Thus except in the case $d =5$, \cite[Theorem 4.1]{BERW} follows from Theorem \ref{thm:general-abelianness-thm}.
\end{remark}

The proof of Theorem \ref{thm:general-abelianness-thm} is quite similar to the proof of \cite[Theorem 4.1]{BERW}. The first step is actually almost the same.

\begin{lem}\label{lem:abeliannessproof-formal}
Let $V: M_0 \leadsto M_1$ be a cobordism such that both boundary inclusions $M_i \to V$ are $2$-connected. Let $g_i \in \Riem^+ (M_i)$ so that there is a right stable $h \in \Riem^+ (V)_{g_0,g_1}$. Then the action of $\Diff_\partial (V)$ on $\Riem^+ (V)_{g_0,g_1}$ is homotopy abelian.
\end{lem}

\begin{proof}
By Theorem \ref{thm:right-and-left-stable}, $h$ is also left-stable. One then uses the formal result \cite[Lemma 4.2]{BERW} in the same way as it was used in the proof of \cite[Theorem 4.1]{BERW}. 
\end{proof}

The next step is to reduce Theorem \ref{thm:general-abelianness-thm} to Lemma \ref{lem:abeliannessproof-formal}. This is significantly more complicated than the corresponding step in \cite[p. 801 f]{BERW}.

\begin{lem}\label{lem:moving-the-2-skeleton-to-boundary}
Let $W$ be a manifold of dimension $d \geq 6$ and assume that $\partial W \to W$ is $2$-connected. Pick a handle decomposition of $W$ without handles of index $\geq d-2$ and let $K$ be the union of all handles of index $\leq 2$. Then there is a collar $C$ of $\partial W$ in $W$ which contains $K$ in its interior. 
\end{lem}

\begin{proof}
We prove the following equivalent statement: if $B$ is an arbitrary collar, then there is an isotopy $f_t$ of $W$, relative to the boundary, so that $f_0=\id$ and $f_1 (K) \subset B$. If that is proven, then $C:= f_1^{-1} (B)$ does the job.

We make use of the Phillips submersion theorem \cite{Phillips}. Let $\Sub (K;W)$ be the space of all submersions $K \to W$ and let $\Epi (TK;TW)$ be the space of bundle epimorphisms $TK \to TW$. The submersion theorem states that the differential map $\Sub (K;W)\to \Epi (TK;TW)$ is a weak homotopy equivalence. The space $\Epi (TK;TW)$ is the space of cross-sections in the fibre bundle $\Fr (TK) \times_{\GL_d (\bR)} \Fr (TW) \to K$, where $\Fr(V)$ denotes the frame bundle of the vector bundle $V$. 

Since the inclusion $B \to W$ is $2$-connected, so is the induced map $\Fr (TB) \to \Fr (TW)$ on frame bundles. Because the handle dimension of $K$ is $\leq 2$, it follows that the map $\Epi (TK;TB) \to  \Epi (TK; TW)$ is $0$-connected, and by the submersion theorem, $\Sub (K;B) \to \Sub (K;W)$ is $0$-connected.

This implies the existence of a regular homotopy $h_t$ from the inclusion map $h_0:K \to W$ to a submersion $h_1:K \to B$. Since $d=\dim (W) \geq 6$, we can assume, by general position, that $h_t$ embeds a neighbourhood $U \subset K$ of the cores of the handles of $K$. There is an embedding $e:K \cong L \subset U$ which is isotopic (as a map $K \to K$) to the identity. 

By the isotopy extension theorem \cite[Theorem 8.1.4]{Hirsch}, we find an isotopy $f_t$ of $W$, starting with the identity and such that $f_t \circ e = h_t \circ e$. Finally, the embeddings $K  \subset W$ and $K \stackrel{e}{\to } K \subset W$ are ambiently isotopic, so that altogether, the inclusion $K \to W$ is ambiently isotopic to an embedding into the collar $B$.
\end{proof}

\begin{proof}[Proof of Theorem \ref{thm:general-abelianness-thm}]
Pick a handlebody decomposition of $W$ using only handles of index $\leq d-3$. Let $K$ be the union of all handles of index $\leq 2$, and let $V:= W \setminus \inter{K}$. Let $M_0:= \partial K$ and $M_1 := \partial W$, so that $V: M_0 \leadsto M_1$. Observe that both inclusions $M_i \to V$ are $2$-connected. According to Lemma \ref{lem:moving-the-2-skeleton-to-boundary}, there is a collar $C$ of $\partial W$ in $W$ containing $K$. Let $\Diff (W,C) \subset \Diff_{\partial} (W)$ be the subgroup of diffeomorphisms, fixing $C$. The inclusion $\Diff (W,C) \to \Diff_{\partial} (W)$ is a weak equivalence, and it factors through the group $\Diff (W,K) \cong \Diff_{\partial} (V)$. Therefore, the map $\Diff_{\partial}(V) \to \Diff_{\partial} (W)$ extending diffeomorphisms as the identity over $K$, induces a surjection on path components. Therefore, it is enough to show that the image of $\pi_0(\Diff_{\partial} (V))\to \pi_0(\hAut (\Riem^+ (W)_g))$ is an abelian group. Write $g_1:= g$. 
By Theorem \ref{thm:existence-stable-metrics1}, there is $g_0 \in  \Riem^+ (M_0)$, and a right stable $h' \in \Riem^+ (K)_{g_0}$. Pick $h'' \in \Riem^+ (V)_{g_0,g_1}$ so that $h' \cup h''$ and $h$ lie in the same path component of $\Riem^+ (W)_{g_1}$, using that $h'$ is right stable. Since $h$ and $h'$ are right stable, so is $h''\in \Riem^* (V)_{g_0,g_1}$, by Lemma \ref{lem:2-out-of-three}.

Because $h'$ is right stable, the map $\mu (h',\_) : \Riem^+ (V)_{g_0,g_1} \to \Riem^+ (W)_{g_1}$ is a weak equivalence, and since $\mu (h',\_)$ is $\Diff_{\partial }(V)$-equivariant, 
we have reduced the problem to showing that $\Diff_{\partial}(V)$ acts on $\Riem^+ (V)_{g_0,g_1}$ through an abelian group. But $h'' \in \Riem^+ (V)_{g_0,g_1}$ is stable, so that by Lemma \ref{lem:abeliannessproof-formal}, the action of $\Diff_\partial (V)$ on $\Riem^+ (V)_{g_0,g_1}$ is homotopy abelian. 
\end{proof}

\begin{cor}\label{cor:to-abelianness}
Let $W:S^{d-1} \leadsto M$ be a cobordism, $d \geq 6$ and assume that $M \to W$ is $2$-connected. Let $g \in \Riem^+ (M)$ be such that there exists a right stable $h \in \Riem^+ (W)_{g_{\round},g}$. Then the action of $\Diff_\partial (W)$ on $\Riem^+ (W)_{g_\round,g}$ is homotopy abelian.
\end{cor}

\begin{proof}
The cobordism $V= D^d \cup W: \emptyset \leadsto M$, with the metric $h':g_\tor^d \cup h$, satisfies the hypotheses of Theorem \ref{thm:general-abelianness-thm}. Let $f_0,f_1 \in \Diff_\partial (W)$ and let $\bar{f_i}$ be the extension to $V$ by the identity. The diagram
\[
 \xymatrix{
 \Riem^+ (W)_{g_\round,g} \ar[r]^{f_i^*} \ar[d]^{\mu(g_\tor^d,\_)} &  \Riem^+ (W)_{g_\round,g}\ar[d]^{\mu(g_\tor^d,\_)} \\
 \Riem^+ (V)_g \ar[r]^{\bar{f_i}^*} & \Riem^+ (V)_g
 }
\]
commutes and the vertical maps are homotopy equivalences. By Theorem \ref{thm:general-abelianness-thm}, $\bar{f_0}^*$ and $\bar{f_1}^*$ commute up to homotopy, and hence so do $f_0^*$ and $f_1^*$.
\end{proof}

\subsection{Starting the proof of Theorem \ref{thm:factorization}}

Let us first recall the set-up of Theorem \ref{thm:factorization}. We are given a $2n$-manifold $W$ such that $(W, \partial W)$ is $(n-1)$-connected, and we form the second stage
$$\tau : W \overset{\ell_W}\lra X \overset{\theta}\lra BO(2n)$$
of the Moore--Postnikov tower for the map $\tau$ classifying the tangent bundle of $W$. The map $\theta$ classifies a vector bundle $\theta^*\gamma_{2n} \to X$, and a \emph{$\theta$-structure} on a $2n$-dimensional manifold $U$ is a bundle map $\hat{\ell}_U : TU \to \theta^*\gamma_{2n}$. Similarly, a \emph{$\theta$-structure} on a $(2n-1)$-manifold $P$ is a bundle map $\hat{\ell}_P : \epsilon^1 \oplus TP \to \theta^*\gamma_{2n}$; if $P$ is the boundary of a collared manifold $U$, then the collar gives an identification $TU\vert_P = \epsilon^1 \oplus TP$ and hence a $\theta$-structure on $U$ induces one of $P$. We call the underlying map $\ell_U : U \to X$ the \emph{structure map} of the $\theta$-manifold $(U, \hat{\ell}_U)$.

In this section, we construct a certain sequence of manifolds $W_k$ such that $W_0=W$ and such that $B \Diff_\partial (W_k)$ homologically approximates the infinite loop space $\loopinf{}_0 \MT{\theta}$. Let us recall the following

\begin{defn}[Wall \cite{WallFin}]
A space $X$ is of \emph{type $(F_k)$} if there exists a finite complex $K$ and a $k$-connected map $K \to X$.
\end{defn}

Each connected space is of type $(F_0)$, and being of type $(F_1)$ means that $\pi_1 (X)$ is finitely generated. A connected space $X$ is of type $(F_2)$ if and only if $\pi_1 (X)$ is finitely presented and $\pi_2 (X)$ is a finitely generated $\bZ[\pi_1 (X)]$-module. As we have a 2-connected map $\ell_W : W \to X$ from a compact manifold, the space $X$ is of type $(F_2)$. If $X$ happens to be of type $(F_n)$ then the proof of Theorem \ref{thm:factorization} we shall give can be somewhat simplified: many of the constructions in this section are designed to deal with the more complicated case when $X$ is only of type $(F_2)$. 

As we have said, we need to construct a suitable sequence of manifolds. Before we state the result, we introduce some notation. Let $M_p$, $p \in \bZ$, be closed $(2n-1)$-dimensional manifolds and $V_{p,p+1}: M_p \leadsto M_{p+1}$ be cobordisms. We then let, for $p<q$,
\[
V_{p,q}:= V_{p,p+1} \cup V_{p+1,p+2} \cup \cdots \cup V_{q-1,q}
\]
and let $V_{-\infty,q}=\bigcup_{p=-\infty}^{q} V_{p,q}$ be the infinite composition. 

\begin{thm}\label{thm:construction-of-mfd-sequence}
Then there exist $(2n-1)$-dimensional $\theta$-manifolds $M_p$ for $p \in \bZ$ and $\theta$-cobordisms $V_{p,p+1}: M_p \leadsto M_{p+1}$, such that
\begin{enumerate}[(i)]
\item $M_0=\partial W$, and $V_{0,1}$ is the manifold $W \cup_{\partial W} (\partial W \times [0,1] \setminus D^{2n})$, viewed as a $\theta$-cobordism $\partial W \to S^{2n-1}$.
\item For $p \geq 1$, $M_p= S^{2n-1}$ and $V_{p,p+1}= (S^{2n-1} \times [p,p+1]) \# (S^n \times S^n)$.  
\item\label{it:LHSConn} The inclusion $M_p \to V_{p,p+1}$ is $(n-1)$-connected, for all $p$. 
\item The inclusion $M_{p+1} \to V_{p,p+1}$ is $2$-connected, for all $p<0$.
\item The structure map $\ell_{V_{-\infty,p}}: V_{-\infty,p} \to X$ is $n$-connected if $p \geq 0$.
\item Each of the manifolds $M_p$ contains an embedded disc $D^{2n-1}$, and each $V_{p,p+1}$ contains an embedded strip $S_{p,p+1} = D^{2n-1} \times [p,p+1]$ restricting to the embedded discs on both ends. For $p \geq 0$, the disc in $M_p=S^{2n-1}$ is the lower half disc $D_-^{2n-1}$, and the strip in $V_{p,p+1}= (S^{2n-1} \times [p,p+1]) \# (S^n \times S^n)$ is $D_-^{2n-1} \times [p,p+1]$ (and we assume that the connected sum with $S^n \times S^n$ is taken on the upper half part of $S^{2n-1} \times [p,p+1]$). 
\end{enumerate}
\end{thm}

We prepare for the proof of Theorem \ref{thm:construction-of-mfd-sequence} with some basic homotopy theory.

\begin{lem}\label{lem:typef2-contable-pii}\mbox{}
There is a relative $CW$-complex $(C, W)$ and an extension $\ell_C : C \to X$ of $\ell_W : W \to X$ such that
\begin{enumerate}[(i)]
\item $\ell_C$ is $n$-connected, and 
\item there exists a filtration $W=C_0 \subset C_1 \subset C_2 \subset\cdots$ of $C$ such that $C_{m+1} $ is obtained from $C_m$ by attaching a single cell of dimension $3 \leq k_m \leq n$.
\end{enumerate}
\end{lem}

\begin{proof}
First note that the homotopy groups of $X$ are countable, as those of $BO(2n)$ and $W$ are.

For part (i), start by letting $C^{(2)} := W$; the structure map $\ell_{C^{(2)}} := \ell_W : C^{(2)}\to X$ is $2$-connected. Look at the portion 
\[
\pi_3 (X)\lra \pi_3 (X, C^{(2)}) \lra \pi_2 (C^{(2)})
\]
of the long exact sequence for the pair $(X, C^{(2)})$. The two outer homotopy groups are countable, so $\pi_3(X, C^{(2)})$ is too, and hence we can attach countably many $3$-cells to $C^{(2)}$ to arrive at a countable complex $C^{(3)}$ with a $3$-connected map to $X$. Repeating this process up to dimension $n$, we obtain a countable complex $C=C^{(n)}$ with an $n$-connected map $\ell_C : C \to X$.

For part (ii), note that $C$ has countably many $k$-cells for $3 \leq k \leq n$. Now we use the well-known property that the attaching map of each $k$-cell goes into a finite subcomplex of $C^{(k-1)}$. Using this property, we find a bijection from $\bN$ to the set of cells of $C \setminus C^{(2)}$, such that the attaching map of the $m$th cell goes into the subcomplex of $C$ given by the union of the first $m-1$ cells with $C^{(2)}$. Now let $X_m$ be the union of $C^{(2)}$ with the first $m$ cells.
\end{proof}

\begin{lem}\label{lem:attaching-handles-typeFn}
Let $\theta: Y \to BO(2n)$ be a fibration, assume that $Y$ is of type $(F_n)$ and let $k \leq n$. Let $V$ be a $2n$-manifold with a $\theta$-structure $\ell_V: V \to Y$ which is $(k-1)$-connected. Furthermore, assume that the inclusion $M:=\partial V \to V$ is $(n-1)$-connected. Then there is a $\theta$-cobordism $W:M \leadsto M'$ which is a composition of elementary cobordisms of index $\geq k$ and $\leq n$, such that the structure map $\ell_V \cup \ell_W : V \cup_M W \to Y$ is $n$-connected.
\end{lem}

\begin{proof}
It enough to find a $W$  so that $\ell_V \cup \ell_W : V \cup_M W \to Y$ is $k$-connected, by induction. By \cite[Theorem A]{WallFin}, there are finitely many elements $x_1, \ldots,x_r \in \pi_k (Y,V)$ so that attaching $k$-cells to $V$ along all of those $x_i$ and extending the map to $Y$ yields an $n$-connected map $V \cup \bigcup_{i=1}^{r} D^k \to Y$.
The long exact sequence of the triple $M \to V \to Y$ yields the piece
\[
\pi_k (Y, M) \lra \pi_k (Y,V) \lra \pi_{k-1}(V, M)=0.
\]
Hence each of the elements $x_i$ can be represented by a square of the form
\[
 \xymatrix{
 S^{k-1} \ar[r]^{g} \ar[d] & M\ar[d]^{\ell_W|_{M}} \\
 D^k \ar[r]^{f} & Y.
 }
\]
Since $k \leq n$, we can perturb $g$ so that it becomes an embedding, by general position. The normal bundle of $g$ is stably trivial because $\nu_g \oplus TS^{k-1} \oplus \bR \cong g^* \ell_W|_{M}^* \theta^*\gamma_{2n} \cong f^* \theta^*\gamma_{2n}$. Again because $k \leq n$, the normal bundle of $g$ is actually trivial, so that $g$ extends to an embedding $\hat{x}_i : S^{k-1} \times D^{2n-k} \to M$, which may be taken to be disjoint for different $i$. Now attach $k$-handles along these embeddings. The $\theta$-structure on $V$ can be extended over these handles, see e.g. \cite[\S 4.1]{GRW}, and the cobordism $W$ may be taken to be the composition of the elementary cobordisms of these $r$ handle attachments.
\end{proof}

\begin{proof}[Proof of Theorem \ref{thm:construction-of-mfd-sequence}]
For $p \geq 0$, the manifolds $M_p$ and $V_{p,p+1}$ are already prescribed in the Theorem. Note that $V_{0,1} \cong W \setminus D^{2n}$.

Let $\theta_m : Y_m \to X$ be the fibrant replacement of the map $\ell_C\vert_{C_m} : C_m \to X$ from Lemma \ref{lem:typef2-contable-pii} and let $f'_m : Y_m \to Y_{m+1}$ be the map induced by the inclusion $C_m \subset C_{m+1}$. Note that $f'_m$ is a map over $X$ and hence over $BO(2n)$. Furthermore, each $Y_m$ is of type $(F_n)$ and each $f_m$ is $2$-connected. 

The manifold $W$ has a $\theta_0$-structure with $n$-connected structure map $W \to Y_0$, by construction, and $\partial W \to W$ is $(n-1)$-connected. Using the map $f'_0$, we can view $W$ as a $\theta_1$-manifold, and the structure map $W \to Y_1$ is $2$-connected. Applying Lemma \ref{lem:attaching-handles-typeFn}, we find a $\theta_1$-cobordism $V_{-1,0}: M_{-1} \leadsto M_0=\partial W$ such that the structure map $V_{-1,0} \cup_{M_0} W \to Y_1$ is $n$-connected, such that $M_0 \to V_{-1,0} $ is $2$-connected and such that $M_{-1} \to V_{-1,0}$ is $(n-1)$-connected. Continuing in this way, we construct a sequence of cobordisms
\[
\ldots \stackrel{V_{-3,-2}}{\leadsto} M_{-2} \stackrel{V_{-2,-1}}{\leadsto} M_{-1}  \stackrel{V_{-1,0}}{\leadsto} M_0 
\]
such that $V_{-p-1,-p}$ has a $Y_{p+1}$-structure and such that the structure map $V_{p,0} \cup_{M_0} W \to Y_p$ is $n$-connected. This property persists if a disc is removed from the interior of $W$, so we have constructed the required manifolds. 

It remains to embed strips $D^{2n-1} \times [p,p+1]$ into $V_{p,p+1}$, for all $p \leq -1$. This can be done inductively, using the connectivity of the cobordisms with respect to both ends.
\end{proof}

\subsection{Application of high-dimensional Madsen--Weiss theory}

We continue to let $\theta: X\to BO(2n)$ be the second Moore--Postnikov stage of the structure map $W \to BO(2n)$. In Theorem \ref{thm:construction-of-mfd-sequence}, we constructed certain manifolds $M_p$ and cobordisms $V_{p,p+1}: M_p \leadsto M_{p+1}$. 

Next let
$$\ell_{V_{-p,\infty}} : V_{-p,\infty} \overset{\ell^p_{V_{-p,\infty}}}\lra X_p \overset{u_p}\lra X$$
be the $n$th stage of the Moore-Postnikov factorization of the $\theta$-structure on $V_{-p, \infty}$, and define $\theta_p = \theta \circ u_p : X_p \to BO(2n)$. The maps
$$V_{-p,q} \subset V_{-p,\infty} \overset{\ell^p_{V_{-p,\infty}}}\lra X_p \overset{u_p}\lra X$$
form the $n$th stage of the Moore-Postnikov factorization of the $\theta$-structure on $V_{-p, q}$ for any $q \geq 1$. By the naturality of the Moore--Postnikov factorization, we obtain a commutative diagram
\[
\xymatrix{
V_{0,1} \ar[r] \ar[d]^{\ell_0} & V_{-1,1} \ar[r] \ar[d]^{\ell_1} & V_{-2,2} \ar[d]^{\ell_2} \ar[r] &\cdots \ar[r]& V_{-\infty,\infty} \ar[d]^{\ell}\\
X_0 \ar[d]^{\theta_0} \ar[r]^{f_1} & X_1 \ar[r]^{f_2 } \ar[d]^{\theta_1} & X_2 \ar[r] \ar[d]^{\theta_2}& \cdots  \ar[r]& X \ar[d]^{\theta}\\
BO(2n) \ar@{=}[r] & BO(2n)\ar@{=}[r] & BO(2n)\ar@{=}[r] & \cdots \ar@{=}[r] & BO(2n).
}
\]
Note that each space $X_p$ is of type $(F_n)$, as $\ell_p := \ell^p_{V_{-p,\infty}}\vert_{V_{-p,p}} : V_{-p,p} \to X_p$ is an $n$-connected map from a compact manifold, and that $\theta_p$ is $n$-coconnected. Moreover, the natural map $\hocolim_p X_p \to X$ is a weak equivalence and each of the maps $f_p:X_p \to X_{p+1}$ are $2$-connected.

Now let $\Diff_\partial (V_{p,q},S_{p,q})$ be the group of diffeomorphisms of $V_{p,q}$ which are the identity near the boundary and on the embedded strip $S_{p,q} \subset V_{p,q}$. Let $\cI=(\bN_0, \leq)$ and $\cJ=(\bN, \leq)$ as directed sets. 
The classifying spaces of these groups induce a directed system of spaces, indexed by $\cI \times \cJ$, namely
\[
(p,q) \longmapsto B \Diff_\partial  (V_{-p,q},S_{-p,q})
\]
with the maps induced by extending diffeomorphism as the identity. The parametrised Pontrjagin--Thom construction provides us with a map
\[
\alpha_{p,q}:B \Diff_\partial   (V_{-p,q},S_{-p,q}) \lra \Omega_{M_{-p},M_q} \loopinf{-1} \MT{\theta_p}
\]
and the diagram
\[
\xymatrix{
B \Diff_\partial  (V_{-p,q},S_{-p,q}) \ar[r]^{\alpha_{p,q}}\ar[d] & \Omega_{M_{-p},M_q} \loopinf{-1} \MT{\theta_p}  \ar[d]\\
B \Diff_\partial  (V_{-p,q+1},S_{-p,q+1}) \ar[r]^{\alpha_{p,q+1}} & \Omega_{M_{-p},M_{q+1}} \loopinf{-1} \MT {\theta_p},
}
\]
where the right hand vertical is the concatenation with the path induced by $V_{q,q+1}$, commutes\footnote{Strictly speaking this does not make sense: the parametrised Pontrjagin--Thom, or ``scanning", construction depends on choices e.g.\ of tubular neighbourhoods, so in principle just gives a homotopy class of map. There are several standard solutions to this problem. One is that in fact the Pontrjagin--Thom construction depends on a \emph{contractible} space of choices and one can simply build such choices in to a model for the classifying space $B \Diff_\partial (V_{-p,q},S_{-p,q})$ used. Another solution is to replace the spectrum $\MT{\theta_p}$ with a homotopy equivalent spectrum $\mathsf{GRW\theta_p}$ constructed from all manifolds in $\bR^\infty$, not just affine ones: this receives a canonical scanning map when $B \Diff_\partial (V_{-p,q},S_{-p,q})$ is modelled as a space of $\theta_p$-manifolds in $\bR^\infty$. Either solution works for our purposes, though we slightly prefer the second option. An extensive discussion of this option is given Section 2.5 of \cite{JEIndex2}.}. We get a map 
\[
\hocolim_q B \Diff_\partial  (V_{-p,q},S_{-p,q})  \stackrel{\hocolim\limits_q \alpha_{p,q}}{\lra} \hocolim_q \Omega_{M_{-p},M_q} \loopinf{-1} \MT {\theta_p}  \simeq \loopinf{} \MT {\theta_p}.
\]

\begin{thm}\label{thm:acyclic}
This map is acyclic on to the path-component which it hits.
\end{thm}
\begin{proof}
This is simply an application of \cite[Theorem 1.5]{GRWHomStab2}; let us explain how to connect that with the formulation given here. We assume familiarity with the notation of \cite{GRWHomStab2}.

The manifold $V^\circ_{-p,q} := V_{-p,q} \setminus \inter{S_{-p,q}}$, with rounded corners, has boundary $P := \partial V^\circ_{-p,q}$. The inclusion $M_{-p} \subset V_{-p,q}$ is $(n-1)$-connected by Theorem \ref{thm:construction-of-mfd-sequence} (\ref{it:LHSConn}) and induction (if $q \geq 1$). It follows that the inclusions $M_{-p} \setminus \inter{D^{2n-1}} \subset V^\circ_{-p,q}$ and hence $P \subset V^\circ_{-p,q}$ are also $(n-1)$-connected.

The map
$$V_{-p,q}^\circ \subset V_{-p,q} \subset V_{-p, \infty} \overset{\ell^p_{V_{-p,\infty}}}\lra X_p$$
is $n$-connected, and restricts to a $\theta_p$-structure $\hat{\ell}_{P}^{(q)} : \epsilon^1 \oplus TP \to \theta_p^*\gamma_{2n}$ on $P = \partial V_{-p,q}^\circ$. The space $\mathrm{Bun}_{n, \partial}^{\theta_p}(TV_{-p,q}^\circ; \hat{\ell}_{P}^{(q)})$ of $\theta_p$-structures which extend $\hat{\ell}_{P}^{(q)}$ and whose underlying map is $n$-connected is therefore nonempty, and by obstruction theory (using that $(V_{-p,q}, P)$ is $(n-1)$-connected and $\theta_p$ is $n$-coconnected) is hence contractible. Thus
$$B\Diff_\partial (V_{-p,q}, S_{-p,q}) = B\Diff_\partial(V^\circ_{-p,q}) \simeq \mathrm{Bun}_{n, \partial}^{\theta_p}(TV_{-p,q}^\circ; \hat{\ell}_{P}^{(q)}) \hcoker \Diff_\partial(V_{-p,q})$$
is a path component of the space $\mathcal{N}_n^{\theta_p}(P; \hat{\ell}_P^{(q)})$.

The manifold $V^\circ_{-p,q+1}$ is obtained from $V^\circ_{-p,q}$ by gluing on $V^\circ_{q,q+1}$, or in other words by boundary connect-sum with a manifold diffeomorphic to $S^n \times S^n \setminus \inter{D^{2n}}$, so its boundary can also be identified with $P$, with a potentially different $\theta_p$-structure $\hat{\ell}^{(q+1)}_{P}$. (In fact, as $\theta_p$ is once-stable, see \cite[\S 5.1]{GRW}, it can be arranged that $\hat{\ell}^{(q+1)}_{P} = \hat{\ell}^{(q)}_{P}$, but that does not matter for this argument.) This identifies $\hocolim_q B \Diff (V_{-p,q},S_{-p,q})$ with a path-component of the space
$$\hocolim_{q \to \infty} \mathcal{N}_n^{\theta_p}(P ; \hat{\ell}^{(q)}_{P})$$
and \cite[Theorem 1.5]{GRWHomStab2} provides an acyclic map from this space to $\loopinf{} \MT {\theta_p}$. The proof of \cite[Theorem 1.5]{GRWHomStab2} shows that it is indeed the map $\hocolim\limits_q \alpha_{p,q}$ which is acyclic.
\end{proof}

Next, we use naturality of the above discussion with respect to $p$. Namely, the maps $f_p: X_p \to X_{p+1}$ induce maps of spectra $\MT {\theta_p} \to \MT {\theta_{p+1}}$ and the diagram
\[
\xymatrix{
B \Diff_\partial  (V_{-p,q},S_{-p,q}) \ar[r]\ar[d] & \Omega_{M_{-p},M_q} \loopinf{-1} \MT {\theta_p}  \ar[d]\\
B \Diff_\partial  (V_{-p-1,q},S_{-p,q}) \ar[r] & \Omega_{M_{-p-1},M_{q}} \loopinf{-1} \MT {\theta_{p+1}}.
}
\]
commutes\footnote{The discussion of the previous footnote applies here too.}. Here the right hand vertical map is the map induced by $f_p$, followed by path concatenation with the path corresponding to the cobordism $V_{-p-1,-p}$. Altogether, we have two directed systems of spaces indexed by $\cI \times \cJ$, namely
\[
(p,q) \longmapsto B \Diff_\partial  (V_{-p,q},S_{-p,q}) \,\,\text{ and }\,\, (p,q) \longmapsto \Omega_{M_{-p},M_q} \loopinf{-1} \MT {\theta_p},
\]
and the maps $\alpha_{p,q}$ together define a map of directed systems. 

\begin{cor}\label{cor:acyclic-map-in-double-hocolim}\mbox{}
\begin{enumerate}[(i)]
\item The map
\[
\hocolim_{p,q} B \Diff_\partial  (V_{-p,q},S_{-p,q}) \stackrel{\hocolim_{p,q}\alpha_{p,q}}{\lra} \hocolim_{p,q} \Omega_{M_{-p},M_q} \loopinf{-1} \MT {\theta_p} 
\]
is acyclic.
\item There is a weak equivalence
 \[
\hocolim_{p,q} \Omega_{M_{-p},M_q} \loopinf{-1} \MT {\theta_p}  \lra \loopinf{}_0 \MT{\theta}.
\]
\item The diagonal map 
\[
\hocolim_p B \Diff_\partial (V_{-p,p},S_{-p,p}) \lra \hocolim_{p,q} B \Diff_\partial (V_{-p,q},S_{-p,q})
\]
is a weak equivalence.
\item\label{it:FinalAcyclicMap} There is an acyclic map 
\[
\hocolim_p B \Diff_\partial (V_{-p,p},S_{-p,p}) \lra \loopinf{}_0 \MT{\theta}.
\]
\end{enumerate}
\end{cor}

\begin{proof}
Part (i) follows from the fact the maps $\hocolim_q \alpha_{p,q}$ are acyclic (by Theorem \ref{thm:acyclic}) and that directed (homotopy) colimits of acyclic maps are acyclic (which follows from the homological characterization of acyclic maps \cite[\S 1]{HH}).

Part (ii) follows from the fact that $\hocolim_p X_p \to X$ is a weak equivalence, which carries over to the Madsen--Tillmann--Weiss spectra and their infinite loop spaces, and the fact that path spaces (when non-empty) are homotopy equivalent to loop spaces. 

Part (iii) follows from the fact that $p \mapsto (p,p) : \bN \to \cI \times \cJ$ is cofinal, and part (iv) follows by combining all the previous parts.
\end{proof}

\subsection{Finishing the proof of Theorem \ref{thm:factorization}}

After all these preliminaries, we can now prove Theorem \ref{thm:factorization} by the same sort of arguments as those in \cite[\S 4]{BERW}. Because the argument is so similar, we shall be quite brief. We use the notation introduced in the previous section. 

\begin{lem}\label{lem:proof-fact-2}
There are psc metrics $g_p \in \Riem^+ (M_p)$ and $h_{p,p+1} \in \Riem^+ (V_{p,p+1})_{g_p,g_{p+1}}$ such that
\begin{enumerate}[(i)]
\item $h_{p,p+1}$ is stable for all $p \neq 0$,
\item $h_{0,1}$ is left stable.
\end{enumerate}
\end{lem}

\begin{proof}
Let $p\geq 1$. Then $M_p=S^{2n-1}$ and $V_{p,p+1} = S^{2n-1} \times [p,p+1] \sharp (S^n \times S^n)$. For those $p$, put $g_p := g_{\round} \in \Riem^+ (M_p)=\Riem^+ (S^{2n-1})$. For $p \geq 1$, we let $h_{p,p+1}\in \Riem^+ (V_{p,p+1})_{g_{p},g_{p+1}}$ be a metric which corresponds to $g_p +dt^2$ under the surgery equivalence $\Riem^+ (V_{p,p+1})_{g_{p},g_{p+1}} \simeq \Riem^+ (S^{2n-1} \times [p,p+1])_{g_\round,g_\round}$ (exactly as in \cite[Proposition 4.8]{BERW}). By Lemma \ref{lem:stability-under-surgery}, $h_{p,p+1}$ is stable.

Let $p=0$. Recall that $W= V_{0,1} \cup D^{2n}$. By hypothesis, there is a psc metric $g_0 \in \Riem^+ (M_0)$, and a \emph{left} stable $h \in \Riem^+ (W)$ (here we view $W$ as a cobordism $M_0 \leadsto \emptyset$, and turning the direction of a cobordism turns right stable into left stable psc metrics). The torpedo metric on $D^{2n}: S^{2n-1}\leadsto \emptyset$ is left stable. Hence there is a metric $h_{0,1} \in \Riem^+ (V_{0,1})_{g_0,g_1}$ such that $h_{0,1} \cup g_{\tor} \in \Riem^+ (W)_{g_0}$ lies in the same component as $h$. By Lemma \ref{lem:2-out-of-three}, $h_{0,1}$ is left stable. 
 
Let $p<0$. Since the cobordism $V_{p,p+1}$ is $2$-connected with respect to either end, the existence of the metrics $g_p$ and $h_{p,p+1}$ with the desired properties follows by a repeated application of Corollary \ref{cor:right-and-left-stable}. 
\end{proof}

\begin{lem}\label{lem:proof-fact-3}
The action of $\Diff_\partial (V_{-p,q}) $ on $\Riem^+ (V_{-p,q})_{g_{-p},g_q}$ is homotopy abelian, for $p \geq 0$ and $q>0$.
\end{lem}

\begin{proof}
This is clear from Corollary \ref{cor:to-abelianness} and Lemma \ref{lem:proof-fact-2}.
\end{proof}

The rest of the argument is as in \cite[\S 4]{BERW} and will only be sketched. Let us write
$$B_0 := B \Diff_\partial  (V_{0,1}, S_{0,1}); \,\,\quad T_0 := E \Diff_\partial  (V_{0,1}, S_{0,1}) \times_{\Diff_\partial  (V_{0,1}, S_{0,1})} \Riem^+ (V_{0,1})_{g_{0},g_1}$$
and, for $p \geq 1$,
\begin{align*}
B_{p} &:= B \Diff_\partial  (V_{-p,p}, S_{-p,p})\\
T_p &:= E \Diff_\partial  (V_{-p,p}, S_{-p,p}) \times_{\Diff_\partial  (V_{-p,p}, S_{-p,p})} \Riem^+ (V_{-p,p})_{g_{-p},g_p}.
\end{align*}
Gluing on the psc manifolds $(V_{p,p+1}, h_{p,p+1})$ and $(V_{-p-1,-p}, h_{-p-1,p})$ gives commutative and homotopy cartesian diagrams
\[
\xymatrix{
T_{p}\ar[r] \ar[d] & T_{p+1} \ar[d]\\
B_{p}\ar[r] & B_{p+1},
}
\]
and passing to the homotopy colimit we obtain a fibration 
$$T_\infty \lra B_\infty := \hocolim_{p} B_{p}$$
from which the fibration $T_{0} \to B_{0}$ is pulled back (up to homotopy). 
The same obstruction theoretic argument as in \cite{BERW}, using Lemma \ref{lem:proof-fact-3}, applied to the acyclic map $\hocolim_{p} B_{p} \to \loopinf{}_0 \MT{\theta}$ of Corollary \ref{cor:acyclic-map-in-double-hocolim} (\ref{it:FinalAcyclicMap}), produces a fibration $T_\infty^+ \to \loopinf{}_0 \MT{\theta}$ which pulls back to $T_\infty \to B_\infty$ and hence to $T_0 \to B_0$. This finishes the proof of Theorem \ref{thm:factorization}.

\section{The secondary Rosenberg index}\label{sec:indextheory}

In this section, we explain how to extend the results of \cite[\S 3]{BERW} to the case of a nontrivial fundamental group. In order to avoid repetitions of large portions of loc.\ cit., we only explain the differences. We will then explain how, together with Theorem \ref{thm:factorization}, these results imply Theorem \ref{main-theorem-mapintoRiem}. 
The reader is warned that the following pages are not meant to be understandable without reference to \cite{BERW}. 

\subsection{The Rosenberg--Dirac operator} 

Let $G$ be a discrete group and let $M$ be a Riemannian spin manifold equipped with a map $M \to BG$ (or equivalently with a $G$-Galois cover). With these data, Rosenberg \cite{RosNovI} associated a certain Dirac operator on an infinite dimensional bundle over $M$. 

Before we recapitulate the construction, we recall the notion of a \emph{Real} $\mathrm{C}^*$-algebra \cite{Kasp}. This is a $\mathrm{C}^*$-algebra $\gA$ over the complex numbers, together with a complex-antilinear automorphism $a \mapsto \overline{a}$ of order $2$. Important examples are: $\gR$, which denotes $\bC$ with the complex conjugation; $\gC_0(X)$, the algebra of all continuous functions $f:X \to \bC$ from a locally compact Hausdorff space which vanish at $\infty$, with conjugation $\overline{f}(x):= \overline{f(x)}$. The complexification $\Cl^{p,q}$ of the real Clifford algebra\footnote{The first $p$ generators have square $-1$, the last $q$ generators have square $+1$.} $\mathrm{Cl}^{p,q}$ is also important for us. 

Let $\bC[G]$ be the complex group ring, with the involution $(\sum_g a_g g)^*:= \sum_g \overline{a_g} g^{-1}$. The regular representation of $G$ on $L^2 (G)$ induces an injective ring homomorphism $\rho:\bC[G] \to \mathrm{Lin} (L^2 (G))$ which preserves $*$, and we define $\norm{x}_{\mathrm{r}}:= \norm{\rho(x)}$. The \emph{reduced group $\mathrm{C}^*$-algebra} $\cstarred (G)$ is the completion of $\bC[G]$ with respect to the norm $\norm{\_}_r$. The \emph{maximal} (or full) group $\mathrm{C}^*$-algebra $\cstarmax(G)$ is obtained by completing $\bC[G]$ with respect to the norm $\norm{x}_{\mathrm{m}} := \sup_{\lambda} \norm{\lambda (x)}$, where $\lambda$ runs over all unitary representations of $G$ on Hilbert spaces \cite[3.7.4]{HR}. If $G$ is countable, then both group $\mathrm{C}^*$-algebras are separable. The conjugation $\overline{\sum_g a_g g}:= \sum_g \overline{a_g} g$ on $\bC[G]$ extends to Real structures on $\cstarred(G)$ and $\cstarmax(G)$. We will write $\cstar(G)$ in all statements which apply to both $\cstarred (G)$ and $\cstarmax (C)$.

We consider $\cstar(G)$ as a right module over itself; the formula $\scpr{x,y}:= x^* y$ turns $\cstar(G)$ into a Hilbert-$\cstar(G)$-module. 
The unitary group $U(\cstar(G))$ contains $G$ as a subgroup; hence $G$ acts by left-multiplication on $\cstar(G)$ by $\cstar(G)$-linear operators preserving the inner product. The \emph{Mishchenko--Fomenko line bundle} is the bundle 
\[
\cL_G:=EG \times_G \cstar(G) \lra BG
\]
of (rank $1$ free) Hilbert modules. Let $(M,g)$ be a Riemannian spin manifold of dimension $d$, with spinor bundle $\spinor_M$. This is a $\Cl(TM \oplus \bR^{0,d})$-module (in the terminology of \cite[\S 3.1.3]{BERW}). If in addition $M$ is equipped with a map $f:M \to BG$, then
\[
\spinor_M \otimes f^* \cL_G \lra M
\]
is a bundle of $\cstar(G)$-Hilbert modules (projective, of finite rank) with $\cstar(G)$-valued inner product $\scpr{-,-}$, and has a compatible action of $\Cl(TM) \otimes \Cl^{0,d}$. The spinor bundle $\spinor_M$ inherits a connection from the Levi-Civita connection on $M$, and $f^* \cL_G$ has a natural flat connection, so $\spinor_M \otimes f^* \cL_G$ has the tensor product connection. The \emph{Rosenberg--Dirac operator} or \emph{$G$-Dirac operator} 
$$\Dir_{f}=\Dir_{f,g}: \Gamma (M;\spinor_M \otimes f^* \cL_G) \to \Gamma (M;\spinor_M \otimes f^* \cL_G)$$
is defined by the classical formula \cite[\S II.5]{SpinGeometry}. The Schr\"odinger--Lichnerowicz formula
\[
\Dir_{f,g}^2 = \nabla^* \nabla + \frac{1}{4} \scal (g)
\]
still holds, and relates $\Dir_{f,g}$ to positive scalar curvature on $M$. The $G$-Dirac operator is formally self-adjoint, odd with respect to the grading on $\spinor_M \otimes f^*\cL_G$ and $ \cstar(G)$-linear. It \emph{anticommutes} with the action of $\Cl^{0,d}$.
\begin{remark}
In \cite{BERW}, we used a slightly different setup in which the Dirac operator was \emph{$\Cl^{d,0}$-linear}. The translation between these is explained in \cite[p. 773]{BERW}. 
\end{remark}

We can extend this construction to the family case, as discussed in \cite[\S 3.2.2]{BERW}. 
Assume that $\pi:E\to X$ is a bundle of $d$-dimensional compact spin manifolds with fibre $W$ equipped with Riemannian metrics and that $f: E \to BG$ is a fixed map. Assume that the boundary bundle is trivial, $\partial E = X \times \partial W$. As in \cite[\S 3.2]{BERW}, we form the elongation $\hat{E}$ by adding $\partial E \times [0, \infty)$, and extend the spin structure, Riemannian metric and map to $BG$ in the obvious fashion. We obtain a bundle of noncompact manifolds with cylindrical ends, again denoted by $E$. We denote the fibres of $\pi$ by $E_x$, the Riemannian metrics on $E_x$ by $g_x$ etc.

Let us now turn to the analytical properties of the $G$-Dirac operator, parallel to \cite[\S 3.2.3]{BERW}. We deviate from the setting used in \cite{BERW} and rely on the analytical results proven in \cite{JEIndex1}. Instead of the Hilbert bundles in \cite{BERW}, we use \emph{continuous fields of Hilbert-$\cstar(G)$-modules}. This notion is defined in \cite[\S 2.1]{JEIndex1} and is a straightforward adaptation of the notion of a continuous field of Hilbert spaces from \cite{DixDou}. For each $x \in X$, the space $\Gamma_c (E_x; \spinor_x \otimes \cL_{f,x})$ of smooth, compactly supported sections has an $\cstar(G)$-valued inner product 
\[
 \scpr{s,t}:= \int_{E_x} \scpr{s(y),t(y)} d \vol_{E_x} (y)
\]
with completion $L^2 (E_x; \spinor_x \otimes \cL_{f,x})$, a countably generated Real graded Hilbert $\cstar (G)$-module with an action of $\Cl^{0,d}$. The family $(L_x^2 (E_x;\spinor_x \otimes \cL_{f,x}))_{x \in X}$ assembles to a continuous field of Hilbert-$\cstar(G)$-modules $L^2_X (E;\spinor_E \otimes f^* \cL_G)$; see \cite[Definition 2.12]{JEIndex1} for further details. 
The individual $G$-Dirac operators $\Dir_{f,x}$ on $E_x$ are unbounded symmetric operators, and together they form an unbounded operator family $\Dir_f$, in the sense of \cite[\S 2.2]{JEIndex1}. By \cite[Theorem 1.14 and Example 2.28]{JEIndex1}, the closure of this operator family is self-adjoint in the sense of \cite[Definition 2.26]{JEIndex1}. The point here is that the projection $p:\partial E \times [0,\infty) \to [0,\infty)$ can be extended to what is called a ``coercive function'' such that $[\Dir_f,p]$ is bounded. This is a generalization of the classical result \cite[Proposition 10.2.10]{HR}. Therefore, we can use functional calculus \cite[\S 2.3]{JEIndex1} and form the bounded transform $\normalize{\Dir_f}$. This is a bounded self-adjoint operator family on $L^2_X(E;\spinor_E \otimes f^* \cL_G)$.

\begin{prop}\label{prop:dirac-is-fredholm}\cite[Theorem 2.41 and Lemma 2.42]{JEIndex1}
If the scalar curvature of $E\to X$ is positive on the ends, then $F=(\normalize{\Dir_{f,x}})_{x \in X}$ is a Fredholm family on $L^2_X(E;\spinor_E \otimes f^* \cL_G)$ (in the sense of \cite[Definition 2.17]{JEIndex1}). If the scalar curvature is everywhere positive over a closed subspace $Y \subset X$, then $F|_Y$ is invertible.
\end{prop}

\subsection{\texorpdfstring{$K$}{K}-theory with coefficients in a \texorpdfstring{$\mathrm{C}^*$}{C*}-algebra}

The framework for topological $K$-theory that we use is different from that in \cite[\S 3.1]{BERW} and is that developed in \cite[\S 3]{JEIndex1}. Let us quickly recall the definition.  

\begin{defn}
Let $\gA$ be a Real graded $\mathrm{C}^*$-algebra (the only relevant example for us is $A= \cstar(G)$) and let $(X,Y)$ be a space pair. The group $KO^{-d}(X,Y;\gA)$ is the group of equivalence classes of tuples $(H,\iota,c,F)$, where
\begin{enumerate}[(i)]
 \item $H$ is a continuous field of Real Hilbert-$\gA$-modules on $X$, with grading $\iota$, 
 \item $c$ is a $\Cl^{0,d}$-structure on $H$ (see \cite[Definition 3.1]{JEIndex1}),
 \item $F$ is a self-adjoint graded Fredholm family on $H$, and $c(v)F+Fc(v)=0$ for all $v \in \bR^{d} \subset \Cl^{0,d}$, such that
 \item $F|_Y$ is invertible. 
\end{enumerate}
Two such tuples are equivalent if they are concordant. The group structure is given by direct sum.
\end{defn}

When $(X,Y)$ is a compact Hausdorff pair, the group $KO^{-d} (X,Y;\gA)$ is naturally isomorphic to Kasparov's Real $KK$-groups
\begin{align*}
 KK(\Cl^{0,d},\gC_0 (X-Y) \otimes \gA) &\cong  KK(\gC,\gC_0 (X-Y) \otimes \gA \otimes \Cl^{d,0})\\
 &\cong KK (\gC,\gC_0( (X-Y)\times \bR^d) \otimes \gA)
\end{align*}
by \cite[Proposition 3.12]{JEIndex1}. The functor $(X,Y) \mapsto KO^{-d}(X,Y;\gA)$ is the degree $-d$ part of the cohomology theory represented by the real $K$-theory spectrum of the graded $\mathrm{C}^*$-algebra $\gA$ (the values for positive degrees can be defined using the Clifford algebra $\Cl^{d,0}$ instead of $\Cl^{0,d}$). As in \cite{BERW}, we shall represent elements in $KO^{-d} (X,Y;\gA)$ by maps of pairs
\[
(X,Y) \lra  (\loopinf{+d} \KO (\gA), *).
\]
Hence if $\pi: E \to X$ is a bundle of spin manifolds with cylindrical ends, equipped with a map $f:E \to BG$ and a fibrewise Riemannian metric which is cylindrical over the ends and has positive scalar curvature, then the $G$-Dirac operators $\Dir_{f,g}$ define an element
\[
 \ind (\Dir_{f,g})  = [L^2_X (E;\spinor_E \otimes f^* \cL_G),\iota,c,\normalize{\Dir_{f,g}}    ]\in KO^{-d} (X;\cstar (G))
\]
and if the scalar curvature is positive over a subspace $Y \subset X$, then this refines to an element in the relative group $KO^{-d} (X,Y;\cstar (G))$.

\subsection{The index difference}

If $W$ is a $d$-dimensional compact spin manifold with boundary, equipped with a map $f: W \to BG$, and $g\in \Riem^+ (\partial W)$, then we can define the Hitchin version of the index difference 
\[
\inddiff^G: \Riem^+ (W)_g \times \Riem^+ (W)_g \lra \loopinf{+d+1}\KO(\cstar(G)),
\]
analogously to \cite[\S 3.3.1]{BERW}. It should be emphasised that it of course depends on the homotopy class of the map $f$ and not merely on the abstract group $G$.

\begin{remark}\label{rem:spectralflow}
For closed $W$, we also can define the Gromov--Lawson version of the index difference, as in \cite[\S 3.3.2]{BERW}. The main result of \cite{Eb13}, restated as Theorem 3.10 in \cite{BERW} shows that both definitions agree if $G=1$. This was used in \cite{BERW} to derive a detection theorem for odd-dimensional manifolds (Theorem B loc.\ cit.) from the even-dimensional case (Theorem A loc.\ cit). 

The proof given in \cite{Eb13} does not generalise to the case $G \neq 1$. In the forthcoming thesis of Lukas Buggisch, a proof that both versions agree for general $G$ will be given, using Kasparov theory. With the aid of that result, one can prove a version of \cite[Theorem C]{BERW} from Theorem \ref{main-theorem-mapintoRiem}, following exacty the argument in \cite[\S 3.6]{BERW}.
\end{remark}

\subsection{The additivity theorem}

An important ingredient in \cite{BERW} was the additivity theorem for the index \cite[Theorem 3.12]{BERW} (a strengthening of a result by Bunke \cite{Bunke1995}).
For the $G$-index, the additivity theorem continues to hold. When the base space is a point, this was already proven in \cite{Bunke1995}. For the general case, one follows the proof given in \cite{BERW}, replacing the quotations to \cite{Eb13} and \cite{HR} by quoting \cite{JEIndex1}, in particular Proposition 2.34 and 2.36. The proof of \cite[Theorem 3.16]{BERW} (which states an additivity property for the index difference) carries over without change to the case $G \neq 1$.

\subsection{The relative index construction}

The results of \cite[\S 3.5]{BERW} are of formal nature and hold verbatim in the present more general framework, replacing $KO$ by $KO$-theory with coefficients in $\cstar(G)$. 

\subsection{Some words about \texorpdfstring{$KK$}{KK}-theory}

One of the key ingredients in \cite{BERW} was the Atiyah--Singer family index theorem for the Dirac operators. Here, we need the family index theorem for the $\cstar(G)$-valued index, in the real case. While this can certainly be extracted from e.g. \cite{Connes-Skandalis}, we are not aware of a detailed account. Presenting the details would lead us too far away from the main focus of this paper, and we content ourselves with precise statements and an overview of the main steps in the proof, using $KK$-theory. We begin by introducing notation for the Kasparov product.
\begin{enumerate}[(i)]
 \item A homomorphism $\phi:\gA \to \gB$ of $\mathrm{C}^*$-algebras defines a class $[\phi] \in KK (\gA,\gB)$ \cite[Example 17.1.2(a)]{Bla}.
 \item The product $KK(\gA,\gB) \times KK(\gB,\gD) \to KK(\gA,\gD)$ is denoted $(\gx,\gy) \mapsto \gx \komprod \gy$.
 \item The exterior product $KK(\gA,\gB) \times KK(\gA',\gB') \to KK(\gA \otimes \gA', \gB \otimes \gB')$ is denoted $(\gx,\gy) \mapsto \gx \exprod \gy$. 
\end{enumerate}

We will give a homotopy-theoretic formulation of the index theorem, and the formulation uses results by Joachim and Stolz \cite{JS}. They constructed spectra $\bKK (\gA, \gB)$ for each pair of Real graded $\mathrm{C}^*$-algebras and spectrum maps 
\begin{equation}\label{eq:joachimstolzmaps}
\begin{aligned}
\bKK(\gA, \gB) \wedge \bKK(\gB, \gD) &\lra \bKK (\gA, \gD)\\
 \bKK(\gA , \gB) \wedge \bKK(\gA', \gB') &\lra \bKK(\gA \otimes \gA', \gB \otimes \gB'),
\end{aligned}
\end{equation}
as well an equivalence
\begin{equation}\label{eq:joachimstolz-equivalence}
F(\Sigma^\infty X_+, \bKK(\gA, \gB)) \simeq \bKK (\gA,\gC(X) \otimes \gB)
\end{equation}
(the left hand side denotes the function spectrum) for each compact space $X$. The equivalence \eqref{eq:joachimstolz-equivalence} implies that 
\begin{equation}\label{eq:joachimstolz-equivalence2}
[\Sigma^\infty X_+, \bKK(\gA, \gB)] \cong KK (\gA,\gC(X) \otimes \gB).
\end{equation}
In particular $\pi_0 (\bKK(\gA,\gB))\cong KK(\gA,\gB)$, and the maps \eqref{eq:joachimstolzmaps} realise the Kasparov product. Bott periodicity in $KK$-theory implies that there are natural equivalences
\[
 \bKK(\Cl^{p,q} \otimes \gA,\gB) \simeq \bKK(\gA, \Cl^{q,p} \otimes \gB) \simeq \Sigma^{p-q} \bKK(\gA, \gB) 
\]
and hence isomorphisms
\[
KK (\Cl^{p,q} \otimes \gA,\gB) \cong \pi_{q-p}(\bKK (\gA,\gB)).
\]
Moreover, there is a natural equivalence
\[
\bKK(\gR,\gA)\simeq \KO(\gA)
\]
for each graded Real $\mathrm{C}^*$-algebra $\gA$. We further write $\KO(\gR)=\KO$.
The \emph{analytic $K$-homology groups} of a compact Hausdorff space $X$ are
\[
KO_{p-q}^{\an} (X):= KK (\Cl^{p,q} \otimes \gC(X),\gR) \cong \pi_{p-q} (\bKK(\gC(X),\gR)).
\]
For an arbitrary space $Y$, one defines 
\[
RKO_{p-q} (Y) := \colim_{X \subset Y} KO_{p-q}^{\an} (X),
\]
where the colimit runs over all compact subspaces of $Y$. The following result is folklore:

\begin{prop}\label{prop:analytic-vs-homotopical-k-homology}
For CW complexes $Y$, there is a natural isomorphism 
\[
KO_{k} (Y):= \pi_{k} (\Sigma^{\infty} Y_+ \wedge \KO) \stackrel{\cong}{\lra} RKO_{k} (Y)
\]
from the topologically defined $KO$-homology groups $KO_{k}(Y)$.
\end{prop}

\begin{proof}[Sketch of proof]
This can be deduced quickly from \cite{JS}, so we include the proof here. Let $X$ be a finite CW-complex. Under the isomorphism
\[
KK(\gC(X),\gC(X)) \cong \pi_0 (\bKK(\gC(X),\gC(X))) \stackrel{\eqref{eq:joachimstolz-equivalence}}{\cong}
[\Sigma^\infty X_+, \bKK(\gC(X),\gR)],
\]
the identity element $[\id_{\gC(X)}] \in KK(\gC(X),\gC(X))$ corresponds to a map of spectra
\[
u_X:\Sigma^{\infty} X_+ \lra \bKK(\gC(X),\gR),
\]
which depends naturally on $X$. Using the $\KO$-module structure on the $\bKK$-spectra, we obtain a map $v_X:\Sigma^{\infty} X_+ \wedge \KO \stackrel{u_X\wedge \mathrm{Id}}{\to} \bKK(\gC(X),\gR) \wedge \KO \to \bKK(\gC(X),\gR)$. On homotopy groups, $v_X$ yields a natural transformation
\[
t_X:KO_k (X) \lra KO^{\an}_k(X).
\]
The source is excisive by general homotopy theory, and the target by excision in analytic $K$-homology. If $X=*$, then $u_X=1 \in \pi_0 (\KO)=\bZ$, and this shows that $t_*$ is an isomorphism.
Hence $t_X$ is an isomorphism for all finite CW-complexes, and this establishes the claimed result upon taking colimits. 
\end{proof}

The universal Mishchenko--Fomenko line bundle $\cL_G \to BG$ yields a map 
\[
 \cL_G: \Sigma^\infty BG_+ \lra \bKK(\gR, \cstar (G))
\]
of spectra. If $f:E \to BG$ is a map from a compact space, then $\cL_G \circ (\Sigma^\infty f_+) \in [\Sigma^{\infty}E_+, \bKK(\gR, \cstar (G))]$ corresponds to an element $[\cL_{f}] \in KK(\gR,\gC(E) \otimes \cstar(G))$ under the isomorphism \eqref{eq:joachimstolz-equivalence2}. There is the following concrete description for $[\cL_f]$: 
\begin{equation}\label{eq:kk-element-michfom}
 [\cL_{f}] =  [\Gamma (E;f^* \cL_G),\_,0]\in KK( \gR, \gC(E) \otimes \cstar(G))
\end{equation}
(the symbol $\_$ denotes the usual representation of $\gR$ on the sections of $f^* \cL_G$).

\begin{defn}\label{defn:novikov-assembly}
The \emph{Novikov assembly map} is the composition
\begin{equation*}
 \Nov: \bKK(\gR, \gR) \wedge BG_+ \stackrel{\id \wedge \cL_G }{\lra}  \bKK(\gR, \gR) \wedge \bKK(\gR, \cstar (G)) \lra    \bKK(\gR, \cstar (G)).
\end{equation*}
\end{defn}
As explained in the introduction to \cite{JS}, the map $\Nov$ induces the classical assembly map on homotopy groups. More precisely, for each compact subset $Y \subset BG$, we have the class $[\cL|_Y]\in KK(\gC(Y),\cstar (G))$, and we get maps
\begin{align*}
KK (\Cl^{p,q}\otimes \gC(Y) ,\gR)  &\lra KK (\Cl^{p,q},\cstar (G))\\
\gx &\longmapsto ([\id_{\Cl^{p,q}}] \exprod [\cL|_Y]) \komprod (\gx  \exprod [\id_{\cstar(G)}]).
\end{align*}
These maps are natural in $Y$ and hence induce a map on the colimit
\begin{equation}\label{eq:classical-novikov-assembly}
RKO_{p-q}(BG):=\colim_{Y \subset BG} KK (\Cl^{p,q}\otimes \gC(Y) ,\gR) \lra KO_{p-q} (\cstar(G)),
\end{equation}
which is the classical Novikov assembly map.

\subsection{The Atiyah--Singer index theorem}

Let $\pi:E \to X$ be a bundle of \emph{closed} spin manifolds on a compact manifold $X$ ($X$ is allowed to have boundary) and let $f:E \to BG$ be a map. To these data there are associated elements in certain $KK$-groups, besides the class \eqref{eq:kk-element-michfom} given by the Mishchenko--Fomenko line bundle. The ordinary spin Dirac operator defines a class
\begin{equation*}
[\Dir] = [L^2_X (E, \spinor_E), \mu, \normalize{\Dir}] \in KK (\gC (E) \otimes \Cl^{0,d},\gC(X))
\end{equation*}
where $\mu$ is the representation of $\gC(E)$ by multiplication operators, and the $G$-Dirac operator defines a class
\begin{equation*}
\begin{aligned}
{ [\Dir_{f}]} = [L^2_X (E, \spinor_E \otimes f^* \cL_G), \mu, \normalize{\Dir_f}]\\
 \in KK (\gC (E) \otimes \Cl^{0,d},\gC(X) \otimes \cstar (G)).
\end{aligned}
\end{equation*}
The index of $\Dir_{f}$ is recovered from this by the unit homomorphism $u : \gR \to \gC(E)$: we have
\begin{equation*}
 \ind (\Dir_{f}) = (u\exprod [\id_{\Cl^{0,d}}]) \komprod [\Dir_{f}]  \in KK(\Cl^{0,d}, \gC(X) \otimes \cstar (G))=KO^{-d}(X;\gA),
\end{equation*}
essentially by the definition of all these terms (note that the composition product with the class of a homomorphism is easy to compute, and so is the exterior product with the class of the identity). The index theorem describes $\ind (\Dir_{f})$ in topological terms. 

The $K$-theoretic Thom class of a rank $n$ spin vector bundle $V \to X$ gives a map $\kothom{V} : \mathrm{Th}(V) \to \loopinf{-n} \KO$ of spaces, which is adjoint to a map $\kothom{V}:\Th (V) \to \Sigma^n \KO$ of spectra. More generally if $V \to X$ is a stable spin vector bundle of rank $r$, we get a spectrum map 
\[
\kothom{V}: \Th (V) \lra \Sigma^{r} \KO.
\]
For example, if $V$ is the additive inverse of the universal vector bundle on $B \Spin (d)$, the corresponding spectrum map is
\[
\kothom{-d}: \MT{Spin (d)} \lra \Sigma^{-d} \KO.
\]

Let $\pi:E \to X$ be a bundle of $d$-dimensional closed spin manifolds, equipped with a map $f: E \to BG$. The normal bundle $\nu(\pi)$ of $\pi:E \to X$ is a stable vector bundle of rank $-d$, and the Pontrjagin--Thom collapse defines a map
\[
 c: \Sigma^\infty X_+ \lra \Th (\nu(\pi))
\]
from the suspension spectrum of $X$ to the Thom spectrum of the normal bundle. Write
\[
\Delta: \Th (\nu(\pi)) \lra  \Th (\nu(\pi)) \wedge E_+
\]
for the diagonal map. Using the classifying map $E \to B \Spin (d)$ of the vertical tangent bundle of $\pi$ and the map $f:E \to BG$, we obtain a map $v:\Th (\nu(\pi)) \to \MT{Spin (d)}$ and so can form the composition
\[
\alpha_\pi^{ad}: \Sigma^\infty X_+  \stackrel{\Delta\circ c}{\lra} \Th (\nu(\pi)) \wedge E_+ \stackrel{v \wedge f_+}{\lra} \MT{Spin (d)} \wedge BG_+.
\]
and its adjoint
\[
\alpha_\pi: X \lra \loopinf{}(\MT{Spin (d)} \wedge BG_+).
\]
\begin{thm}[The index theorem]\label{thm:indextheorem}
If the base space $X$ is compact, then the composition
\[
\loopinf{}(\Sigma^{-d} \Nov \circ (\kothom{-d} \wedge \id_{BG})) \circ \alpha_\pi: X \lra \loopinf{+d} \KO(\cstar(G))
\]
is equal to $\ind(\Dir_{f})$ under the identification of the set $[X,\loopinf{+d} \KO (\cstar(G))]$ of homotopy classes with $KO^{-d}(X;\cstar(G))$.
\end{thm}

We now give a sketch of the proof. The main part is analytical in nature and is carried out on the level of $KK$-groups. The first step is to relate the index of the $G$-Dirac operator to the $K$-homology class of the ordinary Dirac operator.

\begin{prop}\label{prop:index-of-dirac-with-coefficients}
The relation
\[
 \ind (\Dir_{f}) = ([\cL_f] \exprod [\id_{\Cl^{0,d} }] \komprod ([\Dir] \exprod [\id_{\cstar(G)}]) \in KK (\Cl^{0,d}, \gC(X) \otimes \cstar (G)) 
\]
holds. 
\end{prop}

The case of $X=*$ (and with complex $K$-theory) is done in \cite[Theorem 5.22]{SchickKKL2}. The linear-algebraic modifications to carry out the Real case are clear (everything in loc.cit. is Real). To deal with the family case, one replaces the references in \cite{SchickKKL2} to the analytical parts of \cite{Bunke1995} with references to \cite[\S 2.4]{JEIndex1}.

To describe the proof of the index theorem, we have to give explicit descriptions of the Bott maps and the Thom isomorphism. The \emph{Bott class} $\gb_n \in KK(\gR,\gC_0 (\bR^n) \otimes \Cl^{0,n})$ is 
\begin{equation*}
 \gb_n = [\gC_0 (\bR^n;\Cl^{0,n}),1, \frac{\lambda(x)}{(1+\norm{x}^2)^{\frac{1}{2}}}] \in KK(\gR,\gC_0 (\bR^n) \otimes \Cl^{0,n})
\end{equation*}
where $\lambda$ the action by left-multiplication, and $x$ is the identity function on $\bR^n$. The \emph{inverse Bott class} $\ga_n \in KK(\gC_0 (\bR^n) \otimes \Cl^{0,n},\gR)$ is 
\begin{equation}\label{eq:inverse-bottelement}
\ga_n =  [L^2 (\bR^n; \bS_{n,n}), \mu,\normalize{D}] \in  KK(\gC_0 (\bR^n) \otimes \Cl^{0,n},\gR).
\end{equation}
Here $\bS_{n,n}$ is the exterior algebra $\Lambda^* \bR^n$ with its canonical $\Cl^{n,n}$-structure (see e.g. \cite[Definition 3.2]{JEIndex1}), $\mu$ is the action of $\gC_0 (\bR^n) \otimes \Cl^{0,n}$ on $L^2 (\bR^n;\bS_{n,n})$ which combines pointwise multiplication by real-valued functions on $\bR^n$ and the restriction of the Clifford structure to $\Cl^{0,n}$, and $D$ is the Dirac operator (which is equal to $d+d^*$ in this case). Kasparov proved that
\begin{equation}\label{eq:kasp-bott-periodicity}
\begin{aligned}
 \gb_n \komprod \ga_n &= 1 =[\id_{\gR}] \in KK(\gR,\gR)\cong \bZ\\
  \ga_n \komprod \gb_n &= [\id_{\gC_0 ( \bR^n) \otimes \Cl^{0,n} }] \in KK(\gC_0 (\bR^n) \otimes \Cl^{0,n},\gC_0 (\bR^n) \otimes \Cl^{0,n}),
\end{aligned}
\end{equation}
see \cite[Theorem 7]{Kasp} or \cite[p. 101 ff.]{Echterhoff} for more details. 

There are parametrised version of these classes: let $p:V \to Y$ be a real vector bundle of rank $n$ on a locally compact space $Y$, and let $\Cl (V^-)\to Y$ be the bundle of Clifford algebras of $V^-$ (the Clifford generators have positive square). By $\gGamma_0 (V;\Cl(V^-))$, we denote the $\mathrm{C}^*$-algebra of sections of $p^* \Cl (V^-) \to V$ which vanish at infinity. The above elements generalise to
\begin{equation*}
\begin{aligned}
 \gb_V &\in KK(\gC_0 (Y),\gGamma_0 (V; \Cl(V^-)))\\
 \ga_V &\in KK(\gGamma_0 (V; \Cl(V^-)), \gC_0 (Y)).
\end{aligned}
\end{equation*}
A spin structure on $V$ yields a $KK$-equivalence $\gs_V\in KK(\gGamma_0 (Y;\Cl (V^-)), \gC_0 (Y) \otimes \Cl^{0,n})$. Namely, let $P \to Y$ be the underlying $\Spin (n)$-principal bundle and consider the bundle $P \times_{\Spin (n)} \Cl^{n,n} \to Y$. It has an action of $\Cl(V)$ and one of $\Cl^{0,n}$ which anticommute. Using the grading, one can turn the $\Cl (V)$-action into a $\Cl (V^-)$-action, as in \cite[p. 773]{BERW}. We apply this construction to the pullback $p^* V \to V$ and obtain the Thom class of the spin bundle $V$
\begin{equation*}
 \gt_V :=\gb_V \komprod \gs_{p^*V} \in KK(\gC (Y),\gC_0 (V) \otimes \Cl^{0,n})
\end{equation*}
and the inverse Thom class
\begin{equation*}
\gu_V:= \gs_{p^*V}^{-1} \komprod \ga_V \in KK(\gC_0 (V) \otimes \Cl^{0,n},\gC (Y)).
\end{equation*}
These are mutually inverse $KK$-equivalences. 

Now choose an embedding $E \to X \times \bR^n$ over $X$. Let $V \to E$ be the normal bundle of this embedding, and pick an open embedding $j:V \to X \times \bR^n$ over $X$ as a tubular neighbourhood. Extension by zero gives a homomorphism $j_!: \gC_0 (V) \to \gC_0 (X \times \bR^n)$. 

\begin{prop}\label{proofindextheorem2}
The relation
\[
(\gu_V \exprod [\id_{\Cl^{0,d}}]) \komprod [\Dir] = ([j_!] \exprod [\id_{\Cl^{0,n}}]  ) \komprod (\ga_n \exprod [\id_{\gC(X)}]) \in KK(\gC_0 (V) \otimes \Cl^{0,n} ,\gC(X))
\]
holds. 
\end{prop}

\begin{proof}[Sketch of proof]
By homotopy invariance of $KK$-theory, it is enough to prove this equation in $KK(\gC_0 (V_0) \otimes \Cl^{0,n} ,\gC(X))$, where $V_0 \subset V$ is the unit disc bundle. The given Riemannian metric on $E$, a bundle metric on $V$ and a connection on $V$ together define a complete Riemannian metric on the total space $V$. Using the formula for $\gu_V$, one may check that $(\gu_V \exprod [\id_{\Cl^{0,d}}]) \komprod [\Dir]$ is represented by the continuous field of Hilbert spaces $L^2_X (V;\spinor_V)$, with the Clifford action by the (trivial!) vertical tangent bundle of $V$, and the Dirac operator on $V$.
The spin Dirac operator with respect to the euclidean metric on the manifold bundle $X \times \bR^n \to X$ represents the element $\gu_n \exprod [\id_{\gC(X)}]$, by \eqref{eq:inverse-bottelement}. The same is true if we deform the metric, as long as the deformation is constant near infinity. Pick a Riemannian metric on the bundle $X \times \bR^n \to X$ which coincides with the euclidean metric near infinity and with the metric of $V$ on $V_0$. So the two classes we claim are equal are represented by operators which are equal on $V_0$ (but are defined on widely different domains). Using the techniques of \cite[Proposition 10.8.2 and Lemma 10.8.4]{HR} (and the analytical results of \cite[\S 2.4]{JEIndex1} to deal with the family case) finishes the proof.
\end{proof}

Let us introduce more notation. For a rank $n$ spin vector bundle $V \to Y$ on a compact Hausdorff space and $\mathrm{C}^*$-algebras $\gA$, $\gB$, the \emph{Thom isomorphism} is
\begin{align*}
\thom_V:  KK (\gA, \gC(Y) \otimes \gB) &\lra KK(\gA, \gC_0 (V) \otimes \Cl^{0,n} \otimes \gB)\\
 \gx &\longmapsto \gx \komprod (\gt_V \exprod [\id_{ \gB}]).
\end{align*}
The Thom isomorphism of the trivial bundle $Y \times \bR^n \to Y$ is the \emph{Bott map} $\bott_n := \thom_{Y \times \bR^n}$. The inclusion $j:U \to Y$ of an open subspace induces a homomorphism $j_!: \gC_0 (U) \to \gC_0 (Y)$ and we write
\begin{align*}
 j_! : KK(\gC_0 (Y) \otimes \gA, \gB) &\lra KK (\gC_0 (U) \otimes \gA, \gB)\\
 \gx &\longmapsto ([j_!] \exprod [\id_{\gA}]) \komprod \gx.
\end{align*}

With these short notations, Propositions \ref{prop:index-of-dirac-with-coefficients} and \ref{proofindextheorem2}, together with the relations \eqref{eq:kasp-bott-periodicity} and the formal properties of the Kasparov product \cite[\S 18.6 and 18.7]{Bla} imply the formula
\begin{equation}\label{eq:index-theorem-for-dirac-with-coefficients}
 \beta_n (\ind(\Dir_{f})) = j_! (\thom_V ([\cL_f])) \in KK (\Cl^{0,d}, \gC(X \times \bR^n) \otimes \Cl^{0,n} \otimes \cstar (G)).
\end{equation}
Let $c: S^n \wedge X_+ \to \mathrm{Th}(V)$ be the Pontrjagin--Thom collapse, and $\Delta: \mathrm{Th}(V) \to \mathrm{Th} (V) \wedge E_+$ be the diagonal map. The $K$-theory classes $\kothom{V} \in KO^{n-d} (\mathrm{Th}(V),\infty)$ and $\cL_f \in KO (E_+,+;\cstar(G))$ together give a $K$-theory class
\[
\kothom{V} \wedge \cL_f \in KO^{n-d} (\mathrm{Th}(V) \wedge E_+,+;\cstar(G))
\]
which is represented by a based map
\[
\kothom{V} \wedge \cL_f  : \mathrm{Th}(V) \wedge E_+\lra \loopinf{+d-n} \KO (\cstar(G)).
\]
In homotopy theoretic terms, the identification \eqref{eq:index-theorem-for-dirac-with-coefficients} can be reformulated by saying that the adjoint
\[
((\kothom{V} \wedge \cL_f )\circ \Delta \circ c)^{ad}: X \lra \loopinf{+d} \KO (\cstar(G))
\]
of the composition $(\kothom{V} \wedge \cL_f )\circ \Delta \circ c$ represents the index class $\ind(\Dir_f)$. Finally, we let $n$ tend to $\infty$ and use the map $f:E \to BG$ and the classifying map $E \to B\Spin (d)$ of the vertical tangent bundle of $E$ and arrive at the formulation given in Theorem \ref{thm:indextheorem}. This finishes our sketch of the proof of the index theorem.

\subsection{Proof of Theorem \ref{main-theorem-mapintoRiem}}

We now explain how to deduce Theorem \ref{main-theorem-mapintoRiem} from what we have shown so far. Let $f : W \to BG$ and $h \in \cR^+(W)_g$ be as in the statement of Theorem \ref{main-theorem-mapintoRiem}.

By \cite[Lemma 1.3]{WallFin}, as $\pi_1(f) : \pi_1(W) \to G$ is a split surjection from a finitely presented group, $G$ is finitely presented. By embedding a finite presentation $2$-complex for $G$ into $\bR^{2n}$ and taking a regular neighbourhood, we obtain a Spin manifold $P$ with fundamental group $G$, whose structure map $P \to B \Spin (2n) \times BG$ is $2$-connected, and such that $(P, \partial P)$ is $(2n-3)$-connected. (Alternatively, $P$ could be constructed by appealing to Lemma \ref{lem:attaching-handles-typeFn}). We consider $P$ as a cobordism $P : \emptyset \leadsto \partial P$.

By Theorem \ref{thm:existence-stable-metrics1}, there exists a $g_0 \in \Riem^+(\partial P)$ and a right stable $h_0\in \Riem^+ (P)_{g_0}$. The pair $(P,g_0)$ satisfies the hypotheses of Theorem \ref{thm:factorization}, and so we obtain the left half (which is homotopy cartesian) of a diagram
\begin{equation*}
\xymatrix{
\Riem^+ (P)_g \hq \Diff_\partial (P) \ar[r] \ar[d] & T_{\infty}^+ \ar[d] \ar[r] &\ast\ar[d]\\
B \Diff_\partial (P) \ar[r]^-{\alpha_P} & \loopinf{}_0 \MT{Spin (2n)}\wedge BG_+ \ar[r]^-{\eta} & \loopinf{+2n} \KO (\cstarred(G))
}
\end{equation*}
with $\eta=\loopinf{} (\Sigma^{-2n}\Nov \circ (\kothom{-2n} \wedge \id_{BG}))$. The induced map on homotopy fibres is $\inddiff^G_{h_0}$ (the metric $h_0$ enters the construction of the implied nullhomotopy of the composition $\Riem^+ (P)_g \hq \Diff_\partial (P) \to B \Diff_\partial (P) \to \loopinf{+2n} \KO (\cstar(G))$). The right half of the diagram is obtained exactly as in \cite[\S 4.2 and 4.3]{BERW}. The ingredients for this step in \cite{BERW} were the homotopy theoretic formulation of the index theorem \cite[Theorem 3.31]{BERW}, the additivity theorem \cite[Theorem 3.12]{BERW} and the relative index construction \cite[\S 3.5]{BERW}. In the present chapter, we have explained how all these generalise to the case of nontrivial $G$.

Once the commutative diagram above is established, fibre transport for the middle column acting on $h_0 \in \cR^+(P)_{g_0}$ gives a map 
$$\Psi=\Psi_{(P,h_0)}: \loopinf{+1}\MT{Spin (2n)} \wedge BG_+ \lra \Riem^+ (P)_ {g_0}$$
such that
\[
\inddiff^G \circ \Psi \simeq \loopinf{+1} (\Sigma^{-2n}\Nov \circ (\kothom{-2n} \wedge \id_{BG})).
\]
This proves Theorem \ref{main-theorem-mapintoRiem} for the manifold $P$. To prove it for arbitrary $W$ and boundary conditions $g$ we shall embed $P$ into $W$, for which we use the following result.

\begin{thm}\label{embedding-presentation-complex}
Let $W^{2n}$, $n \geq 3$, be a connected compact spin manifold equipped with a map $f:W \to BG$ which is split surjective on $\pi_1$. Let $P^{2n}$ be a spin manifold with a map $h:P \to BG$ which is $2$-connected, and assume that $P$ built from $\emptyset $ by attaching handles of dimension $\leq 2$. Then there exists an embedding $g: P \to W$ preserving the spin structure and such that $f \circ g \simeq h$.
\end{thm}

\begin{proof}
First note that it is really necessary that $\pi_1 (f)$ is split epimorphic. The first part of the proof is pure homotopy theory. Write $B= B \Spin (2n) \times BG$, with universal vector bundle $V\to B$.
Observe that $P$ is homotopy equivalent to a $2$-dimensional complex $K\subset P$. We claim that the lifting problem 
\[
 \xymatrix{
  & W \ar[d]^{f} \\
K \ar@{..>}[ur]^{g} \ar[r]^{h} & B
 }
\]
can be solved (up to homotopy). To do this, let $\sigma: G\to \pi_1 (W)$ be a splitting of the map $\pi_1 (f)$. We identify $G$ and $\pi_1 (K)$ using the isomorphism $\pi_1 (h)$. 
We can assume that $K^{(0)}=*$ and that the attaching maps $S^1 \to K^{(1)}$ for all $2$-cells are pointed.
The complex $K$ is the presentation complex for some finite presentation $\langle S, R\rangle$ of the group $G$, with $S$ indexing the $1$-cells and $R$ indexing the $2$-cells. 
First, we construct the lift $g$ on the $1$-skeleton $K^{(1)}$. Each $1$-cell $\alpha_s$, $s \in S$, is mapped to a loop in $W$ representing $\sigma (s) \in \pi_1 (W)$. This defines $g_1=g|_{K^{(1)}}$.
Let $\beta_r$, $r \in R$ be a $2$-cell with attaching map $\phi_r: S^1 \to K^{(1)}$. Now $g_1 \circ \phi_r: S^1 \to W$ represents the element $\sigma (r) \in \pi_1 (W)$. Since $\sigma$ is a homomorphism, $\sigma(r)=1$, and so $g_1 \circ \phi_r$ is nullhomotopic. Thus we can extend $g_1$ over each $2$-cell of $K$, and this shows the existence of the lift $g$.

Since $P$ is homotopy equivalent to $K$, we obtain a map $g:P \to W$ such that $f \circ g \simeq h$. But then
\[
TP \cong h^* V \cong g^* f^*V \cong g^* TW.
\]
Therefore, $g$ is covered by a bundle map $TP \to TW$. Since $P$ has no closed component, we can apply Phillips' submersion theorem \cite{Phillips}
and find a submersion $g':P \to W$, homotopic to $g$. Thus, we might assume that $g$ is a submersion. 

By general position, we can assume that $g$, restricted to the core of each handle, is self-transverse. As the cores of the handles are at most $2$-dimensional, $g$ will embed all handles, since $\dim (W) \geq 5$. Then $g$ embeds a small neighbourhood $U \subset P$ into $W$. But there is a diffeomorphic copy of $P$ contained in $U$, and hence we found the desired embedding.
\end{proof}

Now take an embedding $g: P \to W$, and view the cobordism $W : \emptyset \leadsto \partial W$ as the composition $\emptyset \stackrel{P}{\leadsto} \partial P \stackrel{K}{\leadsto} \partial W$ of two cobordisms. By hypothesis, $h \in \Riem^+ (W)_{g} \neq \emptyset$. Since $h_0 \in \Riem^+ (P)_{g_0}$ is right-stable, we find a psc metric $h' \in \Riem^+ (K)_{g_0,g}$ such that $h_0 \cup h $ lies in the same component of $\Riem^+ (W)_g$ as $h$. 
The composition 
\[
 \loopinf{+1} \MT{Spin (2n)} \wedge BG_+ \stackrel{\Psi_{(P,h_0)}}{\lra} \Riem^+ (P)_{g_0} \stackrel{\mu (\_,h')}{\lra} \Riem^+ (W)_g
\]
gives the map $\Psi$ whose existence is asserted by Theorem \ref{main-theorem-mapintoRiem}. To compute the composition with the index difference, one uses the additivity theorem for the index, along the same lines as in \cite[Theorem 3.4.11]{BERW}.

\subsection{The Baum--Connes conjecture}

Here we recall the statement of the Baum--Connes conjecture and the relation to the Novikov assembly map, mostly referring to the literature. Here we only use the reduced group $\mathrm{C}^*$-algebra $\cstarred(G)$. For a proper and locally compact $G$-space $X$ and $p,q\in \bN_0$, one defines an analytical assembly map in complex $K$-theory
\[
K_{p-q}^{\an,G} (X):=KK^G (\gC(X)\otimes \Cl^{p,q}, \gC ) \overset{\BC_X}\lra K_{p-q} (\cstarred(G)) := KK (\gC \otimes \Cl^{p,q} ,\cstarred(G))
\]
from the $G$-equivariant analytic $K$-homology of $X$ to the $K$-theory of $\cstarred (G)$. If $\underline{E}G$ is the universal $G$-space with finite isotropy, we obtain a map
\[
\BC: RK^G_i (\underline{E} G):= \colim_{X \subset \underline{E}G} K_{i}^G (X)\lra K_{i} (\cstarred(G))
\]
for all $i \in \bZ$. The \emph{Baum--Connes conjecture for $G$} predicts that $\BC$ is an isomorphism for all $i \in \bZ$. The classical Novikov assembly map 
$$\Nov^r=\Nov: KO_i (BG)=RKO_i (BG) \lra KO_i (\cstarred (G))$$
from \eqref{eq:classical-novikov-assembly} has a complex analogue
\[
\Nov^c: K_i (BG) = RK_i(BG) \lra K_i (\cstarred(G)).
\]
\begin{prop}
If $G$ is torsionfree and the Baum--Connes conjecture holds for $G$, then the real Novikov assembly map $\Nov^r: KO_i(BG)\to KO_i (\cstarred(G))$ is an isomorphism.
\end{prop}

\begin{proof}
This is by juxtaposing several results from the literature. If $G$ is torsionfree, then $\underline{E}G = EG$ and so the source of $\BC$ is just $RK^G_i (E G)$. There is a commutative diagram
\[
\xymatrix{
RK^G_i (EG) \ar[d]^{\cong} \ar[r]^{\BC} & K_i (\cstarred(G)) \\
RK_i (BG) \ar[ur]_{\Nov^c} &
}
\]
which was established rigorously in recent work by Land \cite{Land}. Therefore, if $\BC$ is an isomorphism, then so is $\Nov^c$. But there is a Galois descent property: if $\Nov^c$ is an isomorphism (in all degrees), then so is $\Nov^r$, see \cite{BaumKaroubi}, \cite{Kar-descent} and \cite[Theorem 2.14]{SchickKKK} for a particularly simple proof.
\end{proof}

Similarly, rational injectivity results for the complex Baum--Connes assembly map imply rational injectivity results for the real Novikov assembly map \cite[Corollary 2.13 and Theorem 2.14]{SchickKKK}.\mnote

\subsection{Proof of Theorem \ref{thm:StableSurj}}

Recall that we assume that $G$ is a torsionfree group satisfying the Baum--Connes conjecture, and $f: M \to BG$ is a reference map from an even-dimensional manifold such that $f_* : \pi_1(M) \to G$ is split surjective, and we wish to show that
\begin{equation}\label{eq:BottPeriodicIndDiff}
(\inddiff^G_{h_0})_* : \pi_i(\cR^+(M)[B^{-1}]) \lra KO_{i+2n+1}(\cstarred(G))
\end{equation}
is surjective for each $i \geq 0$. By the Baum--Connes conjecture, the map
$$\nu_* : \pi_{i+2n+1}(\KO \wedge BG_+) \lra KO_{i+2n+1}(\cstarred(G))$$
is surjective. As $\pi_{*}(\KO \wedge -)$ is a homology theory, any class is carried on a finite complex, so given a class $x \in KO_{i+2n+1}(\cstar(G))$ there is a map $f : X \to BG$ from a finite CW-complex such that $x$ is in the image of
$$\pi_{i+2n+1}(\KO \wedge X_+) \overset{f_*}\lra \pi_{i+2n+1}(\KO \wedge BG_+) \overset{\nu_*}\lra KO_{i+2n+1}(\cstar(G)).$$

In \cite[Proof of Theorem 5.5]{BERW} there is constructed a class $\mathfrak{b} \in \pi_4(\MT{Spin(4)})$ such that $\widehat{\mathscr{A}}(\mathfrak{b}) = \beta \in ko_8(*)$ is the Bott class. Multiplication by this class gives a map of spectra
$$S^4 \wedge \MT{Spin(d-4)} \overset{\mathfrak{b} \wedge \id}\lra \MT{Spin(4)} \wedge \MT{Spin(d-4)} \overset{\mu}\lra \MT{Spin(d)},$$
which can be iterated and smashed with $X_+$, whence it gives a map
$$\mathfrak{b}^r \cdot - : \Sigma^{4r}\MT{Spin(d-4r)} \wedge X_+ \lra \MT{Spin(d)} \wedge X_+.$$

For parameters $k \geq \ell$, to be tuned later, consider the commutative diagram
\begin{equation*}
\xymatrix{
\pi_{i+2n+8(k-2\ell)+1}(\MSpin \wedge X_+) \ar[r] & \pi_{i+2n+8(k-2\ell)+1} (\ko \wedge X_+) \ar[d]^-{per}\\
\pi_{i+1-8\ell}(\MT{Spin(2n+8(k-\ell))} \wedge X_+) \ar[u]^-{st} 
\ar[r] \ar[d]^-{\mathfrak{b}^{2\ell} \cdot -}& \pi_{i+2n+8(k-2\ell)+1} (\KO \wedge X_+) \ar[d]^-{\beta^{2\ell} \cdot -}_\sim\\
\pi_{i+1}(\MT{Spin(2n+8k)} \wedge X_+) \ar[r] \ar[d]^-{f_*}& \pi_{i+2n+8k+1} (\KO \wedge X_+) \ar[d]^-{f_*}\\
\pi_{i+1}(\MT{Spin(2n+8k)} \wedge BG_+) \ar[r] \ar[d]& \pi_{i+2n+8k+1} (\KO \wedge BG_+) \ar[d]^-{\Nov_*}\\
\pi_i(\cR^+(M \times B^k)) \ar[r] \ar[d]& KO_{i+2n+8k+1} (\cstarred(G)) \\
\pi_i(\cR^+(M)[B^{-1}]) \ar[r]^-{\eqref{eq:BottPeriodicIndDiff}}& KO_{i+2n+1}(\cstarred(G)) \ar[u]_-{\beta^k \cdot -}^-{\sim}.
}
\end{equation*}

The top map is surjective, by the Anderson--Brown--Peterson \cite{ABP} splitting of $\MSpin_{(2)}$, and the Baas--Sullivan description of $ko_*(-)[\tfrac{1}{2}]$ in terms of $\Omega_*^{Spin}(-)[\tfrac{1}{2}]= \Omega_*^{SO}(-)[\tfrac{1}{2}]$, see e.g.\ \cite{Fuehring}. The map $per$ is an isomorphism as long as
$$2n+8(k-2\ell)+1 > \mathrm{dim}(X),$$
by considering the Atiyah--Hirzebruch spectral sequence. The map $st$ is an isomorphism as long as $i+1-8\ell < 0$.

Thus, for the given class $x \in KO_{i+2n+1}(\cstarred(G))$ carried on the finite complex $X$, we may first choose $\ell$ so that $8\ell > i+1$, and then choose $k$ so that $8k > \mathrm{dim}(X)-2n-1+16\ell$ and $k \geq \ell$. Then the commutativity of the above diagram shows that $x$ is in the image of the map \eqref{eq:BottPeriodicIndDiff} as required.

\bibliographystyle{plain}
\bibliography{pscII}

\end{document}